\documentclass[11pt, reqno]{amsart}

\usepackage{amssymb,amsthm,booktabs,graphicx,ifpdf,fullpage}
\usepackage{hyperref}

\newcommand{\mmod}[1]{\;{\rm mod}\;#1}
\newcommand{\mpmod}[1]{\;({\rm mod}\;#1)}
\newcommand{\sumstar}[1]{\sideset{}{^*}\sum_{#1}}
\newcommand{\im}{\text{Im}}
\newcommand{\re}{\text{Re}}
\newcommand{\eps}{\varepsilon}

\newcommand{\rmi}{\mathrm{i}}

\newtheorem{lemma}{Lemma}[section]
\newtheorem{thm}{Theorem}
\newtheorem{proposition}[lemma]{Proposition}
\theoremstyle{definition}
\newtheorem*{acknowledgment}{Acknowledgments}
\newtheorem*{notation}{Notation}

\numberwithin{equation}{section}

\title{Sums of almost equal squares of primes}
\author{Angel Kumchev}
\address{Department of Mathematics\\ Towson University\\ Towson, MD 21252\\ U.S.A.}
\email{akumchev@towson.edu}

\author{Taiyu Li}
\address{School of Mathematics\\ Shandong University\\ Jinan, Shandong 250100\\ P.R. China}
\email{yli@mail.sdu.edu.cn}

\begin{document}

\begin{abstract}
  We study the representations of large integers $n$ as sums $p_1^2 + \dots + p_s^2$, where $p_1, \dots, p_s$ are primes with $| p_i - (n/s)^{1/2} | \le n^{\theta/2}$, for some fixed $\theta < 1$. When $s = 5$ we use a sieve method to show that all sufficiently large integers $n \equiv 5 \pmod {24}$ can be represented in the above form for $\theta > 8/9$. This improves on earlier work by Liu, L\"{u} and Zhan \cite{LiLvZh06}, who established a similar result for $\theta > 9/10$. We also obtain estimates for the number of integers $n$ satisfying the necessary local conditions but lacking representations of the above form with $s = 3, 4$. When $s = 4$ our estimates improve and generalize recent results by L\"{u} and Zhai \cite{LvZh09}, and when $s = 3$ they appear to be first of their kind.
\end{abstract}


\maketitle

\section{Introduction}
\label{secIntrod}

The study of additive representations as sums of squares of primes goes back to the work of Hua~\cite{Hua38}. Define the sets
\begin{align*}
  \mathcal H_3 &= \{ n \in\mathbb{N}: n \equiv 3 \mpmod{24}, \; 5\nmid n\},\\
  \mathcal H_s &= \{ n \in\mathbb{N}: n \equiv s \mpmod{24}\} \qquad (s \ge 4).
\end{align*}
Hua proved that all sufficiently large integers $n \in \mathcal H_5$ can be expressed as sums of five primes and that ``almost all'' integers $n \in \mathcal H_s$, $s = 3, 4$, can be expressed as sums of $s$ squares of primes. Let $E_s(X)$ denote the number of positive integers $n \in \mathcal H_s$, with $n \le X$, that cannot be represented as sums of $s$ squares of primes. As was observed by Schwarz \cite{Schw61}, Hua's method yields the bounds
\[
  E_s(X) \ll X(\log X)^{-A} \qquad (s = 3, 4),
\]
for any fixed $A > 0$. Since the late 1990s, a burst of activity has produced a series of successive improvements on these two bounds, culminating in the estimates (see Harman and Kumchev \cite{HaKu10})
\begin{equation}\label{i.0}
  E_s(X) \ll X^{(5-s)/2-3/20+\eps} \qquad (s = 3, 4),
\end{equation}
for any fixed $\eps > 0$.

In the present note, we study the additive representations of a large integer $n$ as the sum of three, four or five ``almost equal'' squares of primes. Given a large integer $n \in \mathcal H_s$, $s \ge 3$, we are interested in its representations of the form
\begin{equation}\label{eq1.2}
  \begin{cases}
    n = p_1^2 + \dots + p_s^2, \\
    \left| p_j - (n/s)^{1/2} \right| \le H \qquad (j = 1, \dots, s),
  \end{cases}
\end{equation}
where $H = o( n^{1/2} )$. The first to consider this question were Liu and Zhan \cite{LiZh96}, who showed under the assumption of the Generalized Riemann Hypothesis (GRH) that all sufficiently large $n \in \mathcal H_5$ can be represented in the form \eqref{eq1.2} with $s = 5$ and $H = n^{\theta/2}$ for any fixed $\theta \in (0.9,1)$. Shortly thereafter, Bauer \cite{Baue97} proved that the same conclusion holds unconditionally for $\theta \in (1 - \delta, 1)$, where $\delta > 0$ is a (small) absolute constant. Bauer's method is a variant of a method introduced by Montgomery and Vaughan \cite{MoVa75} to study the exceptional set in the binary Goldbach problem. Thus, the value of $\delta$ in Bauer's result is at the same time quite small and rather difficult to estimate explicitly. In 1998, Liu and Zhan \cite{LiZh98} made an important breakthrough in the study of the major arcs in the Waring--Goldbach problem, which allowed them, among other things, to give a simpler proof of Bauer's result with the explicit value $\delta = 0.04$. Their work was followed by a series of successive improvements on the value of $\delta$:
\begin{quote}
  \begin{tabular}{lcl}
    Liu and Zhan \cite{LiZh00}&       & $\delta = 0.0434...$ \\
    Bauer \cite{Baue05}&                & $\delta = 0.0447...$ \\
    L\"{u} \cite{Lv05} & \quad           & $\delta = 0.0571...$ \\
    Bauer and Wang \cite{BaWa06}&     & $\delta = 0.0642...$ \\
    L\"{u} \cite{Lv06}&                  & $\delta = 0.0714...$ \\
    Liu, L\"{u} and Zhan \cite{LiLvZh06}& & $\delta = 0.1$.
  \end{tabular}
\end{quote}
Furthermore, Meng \cite{Meng06} showed that $\delta = 3/29 = 0.1034...$ is admissible under GRH. Our first theorem improves further on these results by adding $\delta = 1/9 = 0.111...$ to the above list.

\begin{thm}\label{thm1}
  All sufficiently large integers $n \in \mathcal H_5$ can be represented in the form \eqref{eq1.2} with $s=5$ and $H = n^{4/9 + \eps}$ for any fixed $\eps > 0$.
\end{thm}

We also obtain bounds for the number of integers $n \in \mathcal H_s$, $s = 3,4$, without representations as sums of $s$ almost equal squares of primes. For $H = o( X^{1/2} )$ and $s = 3, 4$, we define
\begin{equation}\label{i.2}
  E_s(X; H) =
  \#\big\{ n \in \mathcal H_s : |n - X| \le HX^{1/2} \text{ and \eqref{eq1.2} has no solution} \big\}.
\end{equation}
Our interest in $E_s(X; H)$ is twofold. First, a non-trivial bound of the form
\begin{equation}\label{eq1.4}
  E_s(X; X^{\theta/2}) \ll X^{(1+\theta)/2-\Delta},
\end{equation}
for some fixed $\Delta > 0$, implies that almost all integers $n \in \mathcal H_s$ are representable in the form \eqref{eq1.2} with $H = n^{\theta/2}$. Thus, we are interested in bounds of the form \eqref{eq1.4} with $\theta$ as small as possible. Furthermore, given a value of $\theta$ for which we can achieve a bound of the above form, we want to maximize the value of $\Delta$.

When $s = 4$, L\"{u} and Zhai \cite{LvZh09} obtained results in both directions outlined above. First, they proved that \eqref{eq1.4} holds with $s = 4$, $\theta > 0.84$ and some $\Delta = \Delta(\theta) > 0$. Moreover, L\"{u} and Zhai showed that, for $\theta > 0.9$ and some $\eta = \eta(\theta) > 0$,
\begin{equation}\label{eq1.5}
  E_4(X; X^{\theta/2}) \ll X^{\theta/2 - \eta}.
\end{equation}
We remark that this estimate implies the result of Liu, L\"{u} and Zhan \cite{LiLvZh06} on sums of five almost equal squares of primes, because one can combine \eqref{eq1.5} with known results on the distribution of primes in short intervals to deduce a version of Theorem~\ref{thm1} with $H = n^{\theta/2}$. In fact, we use the same observation to deduce Theorem \ref{thm1} from the following theorem, which sharpens and generalizes~\eqref{eq1.5}.

\begin{thm}\label{thm2}
  For any fixed $\theta$ with $8/9 < \theta < 1$ and for any fixed $\eps > 0$, one has
  \begin{equation}\label{eq1.6}
    E_4(X; X^{\theta/2}) \ll X^{(16-11\theta)/14 + \eps}.
  \end{equation}
\end{thm}

When we embarked on this project, one of our goals was to obtain upper bounds of the form \eqref{eq1.4} that are also relatively sharp in the $\Delta$-aspect. In particular, we were interested in bounds in which $\Delta$ is an increasing function of $\theta$ and which approach the best known bounds for $E_s(X)$ when $\theta \uparrow 1$. In the extreme case $\theta = 1 - \eps$, the exponent in \eqref{eq1.6} becomes $5/14 + 2\eps$, which falls just short of~\eqref{i.0} and matches the strength of the second-best bound for $E_4(X)$ obtained by Harman and Kumchev~\cite{HaKu06}. We remark that the choice $\sigma = (2\theta - 1)/7$ in Section \ref{sec6}, which determines the exponent on the right side of~\eqref{eq1.6}, represents a natural barrier for the methods employed in this work. Thus, Theorem~\ref{thm2} is, in some sense, best possible.

We also obtain a small improvement on the first result of L\"{u} and Zhai \cite{LvZh09}. Our next theorem extends that result in two ways: it reduces the lower bound on $\theta$ to $\theta > 0.82$, and it gives an explicit expression for $\Delta$.

\begin{thm}\label{thm3}
  For any fixed $\theta$ with $0.82 < \theta < 1$ and for any fixed $\eps > 0$, one has
  \[
    E_4(X; X^{\theta/2}) \ll X^{1 - 2\sigma + \eps},
  \]
  where $\sigma = \sigma(\theta) = \min(\theta-31/40, (2\theta-1)/8)$.
\end{thm}

The methods used to establish Theorems \ref{thm2} and \ref{thm3} can also be used to establish the following estimate for $E_3(X; X^{\theta/2})$ which, to the best of our knowledge, is the first result of this type for sums of three almost equal squares of primes.

\begin{thm}\label{thm4}
  For any fixed $\theta$ with $0.85 < \theta < 1$ and for any fixed $\eps > 0$, one has
  \begin{equation}\label{eq1.7}
    E_3(X; X^{\theta/2}) \ll X^{1 - \sigma + \eps},
  \end{equation}
  where $\sigma = \sigma(\theta)$ is the function from Theorem \ref{thm3}. Furthermore, when $8/9 < \theta < 1$, one has
  \[
    E_3(X; X^{\theta/2}) \ll X^{(8-2\theta)/7 + \eps}.
  \]
\end{thm}

Unlike the classical Waring--Goldbach problem, the problem considered in this note remains of interest even when the number of variables exceeds five. Indeed, it will take little effort for the reader familiar with the circle method to realize that the work behind Theorem \ref{thm3} can be used to establish a version of Theorem \ref{thm1} with $s = 8$ and $H = n^{0.41+\eps}$. Similarly, the work behind Theorem~\ref{thm4} can be used to establish a version of Theorem \ref{thm1} for six squares. In fact, it is possible to establish such results for all fixed $s \ge 6$. We state those as the following theorem.

\begin{thm}\label{thm5}
  Let $s \ge 6$ and define
  \[
    \theta_s = \begin{cases}
      (1 + 0.775(s-4))/(s-3) & \text{if } 6 \le s \le 16, \\
      19/24                  & \text{if } s \ge 17.
    \end{cases}
  \]
  All sufficiently large integers $n \in \mathcal H_s$ can be represented in the form \eqref{eq1.2} with  $H = n^{\theta_s/2 + \eps}$ for any fixed $\eps > 0$.
\end{thm}

There is an important difference between Theorems \ref{thm3}--\ref{thm5}, on one hand, and Theorems \ref{thm1} and \ref{thm2} (and most other modern results on the Waring--Goldbach problem), on the other. New knowledge of the distribution of primes in short intervals or in arithmetic progressions will have no immediate effect on the quality of Theorems \ref{thm1} and \ref{thm2} (though it may lead to somewhat simpler proofs). In contrast, the lower bounds for $\theta$ in Theorems \ref{thm3} and \ref{thm4} and the choice of $\theta_s$ in Theorem \ref{thm5} are tied directly to results from multiplicative number theory.

\begin{notation}
  Throughout the paper, the letter $\varepsilon$ denotes a sufficiently small positive real number. Any statement in which $\varepsilon$ occurs holds for each positive $\varepsilon$, and any implied constant in such a statement is allowed to depend on $\varepsilon$. The letter $p$, with or without subscripts, is reserved for prime numbers; $c$ denotes an absolute constant, not necessarily the same in all occurrences. As usual in number theory, $\mu(n)$, $\phi(n)$ and $\tau(n)$ denote, respectively, the M\"obius function, the Euler totient function and the number of divisors function. Also, if $n \in \mathbb N$ and $z \ge 2$, we define
  \begin{equation}\label{eq1.8}
    \psi(n,z)= \begin{cases}
      1 & \text{if $n$ is divisible by no prime $p < z$,} \\
      0 & \text{otherwise.}
    \end{cases}
  \end{equation}
  It is also convenient to extend the function $\psi(n, z)$ to all real $n \ge 1$ by setting $\psi(n,z) = 0$ for $n \notin \mathbb Z$. We write $\mathrm e(x)=\exp(2\pi\rmi x)$ and $(a,b)=\gcd(a,b)$, and we use $m \sim M$ as an abbreviation for the condition $M\le m<2M$. We use $\chi(n)$ to denote Dirichlet characters and set $\delta_{\chi} = 1$ or $0$ according as $\chi$ is principal or not. The sums $\sum\limits_{\chi \bmod q}$ and $\mathop{\sum^*}\limits_{\chi \bmod q}$ denote summations over all the characters modulo $q$ and over the primitive characters modulo $q$, respectively.
\end{notation}

\section{Outline of the method}
\label{sec2}

We shall focus on the proof of Theorem \ref{thm2}. Let $\mathcal Z$ denote the set of integers counted by $E_4(X; X^{\theta/2})$. We set
\[
  x = \sqrt{X/4}, \quad \mathcal I = \big( x-x^{\theta}, x + x^{\theta} \big],
\]
and we consider the sum
\[
  R_4(n) = \sum_{\substack{ m_1^2 + \dots + m_4^2 = n\\ m_i \in \mathcal I}} \varpi(m_1) \cdots \varpi(m_4),
\]
where $\varpi$ is the characteristic function of the prime numbers. We note that $R_4(n) = 0$ for all $n \in \mathcal Z$. All prior work on the problem uses the Hardy--Littlewood circle method to analyze $R_4(n)$ (or similar quantities) and to derive bounds for $E_s(X; X^{\theta/2})$. In contrast, we apply the circle method to a related sum, which we construct using Harman's ``alternative sieve'' method \cite[Ch. 3]{Harm07}. Suppose that we have arithmetic functions $\lambda_1, \lambda_2$ and $\lambda_3$ such that, for $m \in \mathcal I$,
\begin{equation}\label{eq2.1}
  \lambda_3(m) \ge \varpi(m) \ge \lambda_1(m) - \lambda_2(m), \quad \lambda_2(m) \ge 0.
\end{equation}
Then
\begin{equation}\label{eq2.2}
  R_4(n) \ge R_4(n; \varpi, \lambda_1) - R_4(n; \lambda_2, \lambda_3),
\end{equation}
where
\[
  R_4(n; \lambda, \nu) = \sum_{\substack{ m_1^2 + \dots + m_4^2 = n\\ m_i \in \mathcal I}} \varpi(m_1) \varpi(m_2) \lambda(m_3)\nu(m_4).
\]
It is convenient to set $\lambda_0 = \varpi$, as we can then work simultaneously with the two sums on the right side of \eqref{eq2.2}.

We apply the circle method to $R_4(n; \lambda_j, \lambda_k)$, $0 \le j,k \le 3$. We have
\begin{equation}\label{eq2.3}
  R_4(n; \lambda_j, \lambda_k) = \int_0^1 f_0(\alpha)^2f_j(\alpha)f_k(\alpha) \mathrm e(-\alpha n) \, d\alpha,
\end{equation}
where
\[
  f_i(\alpha)= f(\alpha; \lambda_i), \quad
  f(\alpha; \lambda) = \sum_{m \in \mathcal{I}} \lambda(m)\mathrm e(\alpha m^2).
\]
Suppose that $\sigma$ is a fixed real number, with $0 < \sigma < \theta/2$. We set
\begin{equation}\label{eq2.4}
  P = x^{2\sigma - \eps}, \quad Q = x^{2\theta}P^{-1}, \quad L = \log x,
\end{equation}
and define
\[
  \mathfrak M(q,a) = \big\{ \alpha \in \mathbb R : |q\alpha-a| \le Q^{-1} \big\}.
\]
The sets of major and minor arcs in the application of the circle method are given, respectively, by
\begin{equation}\label{eq2.5}
  \mathfrak M = \bigcup_{\substack{ 1 \le a \le q \le P\\ (a,q) = 1}} \mathfrak M(q,a)
  \quad \text{and} \quad \mathfrak m = \big[ Q^{-1}, 1 + Q^{-1} \big] \setminus \mathfrak M.
\end{equation}

The most difficult part of the evaluation of the integral in \eqref{eq2.3} is the estimation of the contribution from the minor arcs. In fact, we can control that contribution only on average over $n$. We write
\begin{gather*}
  F(\alpha) = f_0(\alpha)^2(f_0(\alpha)f_1(\alpha) - f_2(\alpha)f_3(\alpha)), \quad Z(\alpha) = \sum_{n \in \mathcal Z} \mathrm e(-\alpha n).
\end{gather*}
In Section \ref{sec3}, we show that
\begin{equation}\label{eq2.6}
  \int_{\mathfrak m} F(\alpha)Z(\alpha) \, d\alpha \ll
  x^{\theta-\sigma + \eps}\big( |\mathcal Z|^{1/2}x^{3\theta/2} + |\mathcal Z|x^{\theta} \big),
\end{equation}
whenever $\lambda_1$ and $\lambda_3$ satisfy the following hypothesis:
\begin{enumerate}
  \item [(i)] $\min(|f_0(\alpha)|, |f_i(\alpha)|) \ll x^{\theta - \sigma + \eps}$ whenever $\alpha \in \mathfrak m$.
\end{enumerate}
In Section \ref{sec6}, we construct sieve functions $\lambda_1, \lambda_2$ and $\lambda_3$ satisfying \eqref{eq2.1} and such that this hypothesis holds for $\lambda_1$ and $\lambda_3$.

Since the $\lambda_i$'s constructed in Section \ref{sec6} have arithmetic properties similar to $\varpi$, one expects that one should be able to evaluate the contribution of the major arcs to right side of \eqref{eq2.3} by a variant of the methods in Liu and Zhan \cite[Ch. 6]{LiuZhan}. However, for technical reasons (see \eqref{eq5.4} below), we need to modify slightly the sums $f_i(\alpha)$, $i \ge 1$, on the major arcs before we can apply those methods. When $\alpha \in \mathfrak M$, we shall decompose $f_i(\alpha)$ as
\begin{equation}\label{eq2.7}
  f_i(\alpha) = g_i(\alpha) + h_i(\alpha),
\end{equation}
with $h_i(\alpha)$ satisfying
\begin{equation}\label{eq2.8}
  h_i(\alpha) \ll x^{\theta - \sigma}.
\end{equation}
We then define the generating functions
\[
  G(\alpha) = f_0(\alpha)^2(f_0(\alpha)g_1(\alpha) - g_2(\alpha)g_3(\alpha)), \quad
  H(\alpha) = F(\alpha) - G(\alpha).
\]
In Section \ref{sec5}, we use the properties of the sieve weights $\lambda_1, \lambda_2$ and $\lambda_3$ to show that if $n \in \mathcal Z$, then
\begin{equation}\label{eq2.9}
  \int_{\mathfrak M} G(\alpha)\mathrm e(-\alpha n) \, d\alpha = K_n x^{3\theta-1}L^{-4}\big( C + O\big( L^{-1} \big)\big),
\end{equation}
where $1 \ll K_n \ll L$ and $C = C(\theta, \sigma)$ is a certain numerical constant related to the construction of the $\lambda_i$'s. Hence, on the assumption that $C > 0$, we have
\begin{equation}\label{eq2.10}
  \int_{\mathfrak M} G(\alpha)Z(\alpha) \, d\alpha \gg |\mathcal Z|x^{3\theta-1}L^{-4}.
\end{equation}
Furthermore, using \eqref{eq2.8} and a variant of the argument leading to \eqref{eq2.6}, we establish the bound
\begin{equation}\label{eq2.11}
  \int_{\mathfrak M} H(\alpha)Z(\alpha) \, d\alpha \ll
  x^{\theta-\sigma + \eps}\big( |\mathcal Z|^{1/2}x^{3\theta/2} + |\mathcal Z|x^{\theta} \big).
\end{equation}

Since $R_4(n) = 0$ when $n \in \mathcal Z$, we obtain from \eqref{eq2.2} that
\[
  \int_0^1 F(\alpha)Z(\alpha) \, d\alpha \le 0.
\]
Combining this inequality and \eqref{eq2.10}, we obtain
\begin{equation}\label{eq2.12}
  |\mathcal Z|x^{3\theta-1}L^{-4} \ll
  \bigg| \int_{\mathfrak m} F(\alpha)Z(\alpha) \, d\alpha + \int_{\mathfrak M} H(\alpha)Z(\alpha) \, d\alpha \bigg|.
\end{equation}
Hence, under the assumptions that $C > 0$ and $\sigma > 1 - \theta$, it follows readily from \eqref{eq2.6}, \eqref{eq2.11} and \eqref{eq2.12} that
\[
  |\mathcal Z| \ll x^{2-\theta-2\sigma+\eps}.
\]
To complete the proof of Theorem \ref{thm2}, we show that $C(\theta, \sigma) > 0$ when $8/9 < \theta < 1$ and $\sigma = (2\theta - 1)/7$.

\section{Proof of Theorem \ref{thm2}: Minor arc estimates}
\label{sec3}

In this section, we assume that the exponential sums $h_i(\alpha)$ satisfy \eqref{eq2.8} and the sieve coefficients $\lambda_1$ and $\lambda_3$ satisfy hypothesis (i) in Section \ref{sec2}, and we deduce inequalities \eqref{eq2.6} and \eqref{eq2.11}.

First, we consider \eqref{eq2.6}. Since the functions $\lambda_i$ which we construct in Section \ref{sec6} are bounded, a comparison of the underlying Diophantine equations yields
\begin{equation}\label{eq3.1}
  I_i = \int_0^1 |f_i(\alpha)|^4 \, d\alpha \ll \int_0^1 \bigg| \sum_{m \in \mathcal I} \mathrm e(\alpha m^2) \bigg|^4 \, d\alpha
  \ll x^{2\theta + \eps},
\end{equation}
where the last inequality follows from \cite[Lemma 4.1]{LiWu08}. Also, using an idea of Wooley \cite{Wool02} (see Liu and Zhan \cite[eq. (8.31)]{LiuZhan}), we have
\begin{equation}\label{eq3.2}
  I_* = \int_0^1 |f_0(\alpha)Z(\alpha)|^2 \, d\alpha \ll |\mathcal Z|x^{\theta+\eps} + |\mathcal Z|^2x^{\eps}.
\end{equation}
Hence, assuming hypothesis (i) for $\lambda_1$ and $\lambda_3$, we can apply H\"older's inequality to show that
\begin{align}
  \int_{\mathfrak m} |F(\alpha)Z(\alpha)| \, d\alpha
  &\ll x^{\theta-\sigma+\eps}I_*^{1/2} \big( I_0^{1/2} + (I_0I_1)^{1/4} + (I_0I_2)^{1/4} + (I_2I_3)^{1/4} \big) \notag \\
  &\ll x^{\theta-\sigma + \eps}\big( |\mathcal Z|^{1/2}x^{3\theta/2+\eps} + |\mathcal Z|x^{\theta+\eps} \big). \label{eq3.3}
\end{align}

The proof of \eqref{eq2.11} is similar. We have
\[
  H(\alpha) = f_0(\alpha)^3h_1(\alpha) + f_0(\alpha)^2(h_2(\alpha)f_3(\alpha) + g_2(\alpha)h_3(\alpha)),
\]
so we may use \eqref{eq2.8}, \eqref{eq3.1}, \eqref{eq3.2} and H\"older's inequality in a similar fashion to \eqref{eq3.3} to establish \eqref{eq2.11}. The only difference is that in the process we also need a bound for the fourth moment of $g_2(\alpha)$ on the major arcs. Using \eqref{eq2.4}, \eqref{eq2.8} and \eqref{eq3.1}, we obtain
\[
  \int_{\mathfrak M} |g_2(\alpha)|^4 \, d\alpha
  \ll \int_0^1 |f_2(\alpha)|^4 \, d\alpha + x^{4\theta - 4\sigma}|\mathfrak M| \ll x^{2\theta + \eps},
\]
which suffices to complete the proof of \eqref{eq2.11}.

\section{Exponential sum estimates}
\label{sec4}

\subsection{Minor arc estimates}
\label{sec4.1}

In this section, we gather the exponential sum estimates needed to establish that the sieve functions satisfy hypothesis (i). Let $\theta$ and $\sigma$ be fixed real numbers, with $2/3 < \theta < 1$ and $0 < \sigma < (3\theta - 2)/6$, and define
\begin{equation}\label{eq4.1}
  Q_0 = x^{3\theta-1-4\sigma}.
\end{equation}
We use Dirichlet's theorem on Diophantine approximation to approximate each real $\alpha$ by a rational number $a/q$, $a, q \in \mathbb Z$, such that
\begin{equation}\label{eq4.2}
  1 \le q \le Q_0, \quad (a,q)=1, \quad \left|\alpha - \frac{a}{q} \right| \le \frac 1{qQ_0}.
\end{equation}
The first two lemmas follow from Propositions A and B in Liu and Zhan \cite{LiZh96} on account of \eqref{eq4.1} and \eqref{eq4.2}.

\begin{lemma}\label{lem4.1}
  Let $2/3 < \theta < 1$ and $0 < \sigma < (2\theta - 1)/6$, and suppose that $\alpha$ has a rational approximation $a/q$ satisfying \eqref{eq4.2} with $q \ge x^{4\sigma}$. Suppose also that the coefficients $\xi_u, \eta_v$ satisfy $|\xi_u| \le \tau(u)^c, |\eta_v| \le \tau(v)^c$. Then
  \[
    \sum_{u \sim U}\sum_{uv \in \mathcal I} \xi_u\eta_v \mathrm e ( \alpha u^2v^2 )
    \ll x^{\theta-\sigma+\varepsilon},
  \]
  provided that
  \[
    x^{1-\theta+2\sigma} \le U \le x^{\theta-4\sigma}.
  \]
\end{lemma}

\begin{lemma}\label{lem4.2}
  Let $2/3 < \theta < 1$ and $0 < \sigma < (3\theta - 1)/6$, and suppose that $\alpha$ has a rational approximation $a/q$ satisfying \eqref{eq4.2} with $q \ge x^{2\sigma}$. Suppose also that the coefficients $\xi_u$ satisfy $|\xi_u| \le \tau(u)^c$. Then
  \[
    \sum_{u \sim U}\sum_{uv \in \mathcal I} \xi_u \mathrm e ( \alpha u^2v^2 )
    \ll x^{\theta-\sigma+\varepsilon},
  \]
  provided that
  \[
    U \le x^{\theta/2 - \sigma}.
  \]
\end{lemma}

In the next lemma, we use Lemmas \ref{lem4.1} and \ref{lem4.2} to derive an estimate for a special kind of double exponential sums which arise in applications of Harman's sieve method.

\begin{lemma}\label{lem4.3}
  Let $\psi(v,z)$ be defined by {\rm (\ref{eq1.8})}. Let $2/3 < \theta < 1$ and $0 < \sigma < (3\theta - 2)/6$, and suppose that $\alpha$ has a rational approximation $a/q$ satisfying \eqref{eq4.2} with $q \ge x^{4\sigma}$. Suppose also that the coefficients $\xi_u$ satisfy $|\xi_u| \le \tau(u)^c$. Then
  \begin{equation}\label{eq4.3}
    \sum_{u \sim U}\sum_{uv \in \mathcal I} \xi_u \psi(v,z) \mathrm e ( \alpha u^2v^2 )
    \ll x^{\theta-\sigma+\varepsilon},
  \end{equation}
  provided that
  \[
    U \le x^{\theta - 4\sigma}, \quad z\le x^{2\theta-1-6\sigma}.
  \]
\end{lemma}

\begin{proof}
  When $U \ge x^{1-\theta+2\sigma}$, the desired result follows from Lemma \ref{lem4.1}. We may therefore suppose that $U \le x^{1-\theta+2\sigma}$. Note that, by the hypothesis on $\sigma$, we then have $U \le x^{\theta/2-\sigma}$. Let
  \[
    \mathcal P(z) = \prod_{p<z}p.
  \]
  Then the left side of \eqref{eq4.3} is
  \[
    \sum_{u\sim U} \sum_{d|\mathcal P(z)} \sum_{dur \in \mathcal I} \mu(d)\xi_u \mathrm e( \alpha d^2u^2r^2 )
    = \Sigma_1 + \Sigma_2 + \Sigma_3,
  \]
  where $\Sigma_1$, $\Sigma_2$ and $\Sigma_3$ are the parts of the sum on the left subject, respectively, to the constraints
  \[
    du \le x^{1-\theta+2\sigma}, \quad x^{1-\theta+2\sigma} < du \le x^{\theta - 4\sigma}
    \quad \text{and} \quad du > x^{\theta - 4\sigma}.
  \]
  By writing $y = du$, we obtain
  \[
    \Sigma_1 \ll L\sup_{Y \le x^{1-\theta+2\sigma}}
    \bigg| \sum_{y \sim Y} \sum_{yr \in \mathcal I} \xi_{y}' \mathrm e( \alpha y^2r^2 ) \bigg|,
  \]
  with coefficients $\xi_y'$ subject to $|\xi_y'| \le \tau(y)^c$. Hence, the desired upper bound for $\Sigma_1$ follows from Lemma \ref{lem4.2}. Similarly, the desired upper bound for $\Sigma_2$ follows from Lemma \ref{lem4.1}. To estimate $\Sigma_3$, we apply the method of proof of \cite[Theorem~3.1]{Harm07} to decompose $\Sigma_3$ into a linear combination of $O(L^2)$ sums of the type appearing in Lemma \ref{lem4.1}. The basic idea is to take out the prime factors of $d$, one by one, until we construct a divisor $k$ of $d$ such that $x^{1-\theta+2\sigma} \le ku \le x^{\theta-4\sigma}$. The hypothesis on $z$ is chosen to ensure that this is possible. Finally, we apply Lemma \ref{lem4.1} to estimate the bilinear sums occurring in the decomposition.
\end{proof}

We also need a variant of the main result in Liu, L\"u and Zhan \cite{LiLvZh06}.

\begin{lemma}\label{lem4.4}
  Let $7/10 < \theta < 1$ and $0 < \sigma < \min\{ (3\theta - 2)/6, (10\theta-7)/15 \}$, and suppose that $\alpha$ has a rational approximation $a/q$ satisfying \eqref{eq4.2} with $q \le x^{4\sigma}$. Suppose also that $\psi$ is a fixed Dirichlet character modulo $r$, $r \le x$. Then
  \[
    \sum_{p \in \mathcal I} \psi(p) \mathrm e( \alpha p^2 )
    \ll rx^{\theta-\sigma+\eps},
  \]
  provided that
  \[
    q + x^{2\theta}|q\alpha - a| \ge x^{2\sigma}.
  \]
\end{lemma}

\begin{proof}
  We need to make a slight adjustment to the proof of \cite[Theorem~1.1]{LiLvZh06}. In place of \cite[eq.~(2.8)]{LiLvZh06}, we have
  \begin{align*}
    \sum_{p \in \mathcal I} \psi(p) \mathrm e( \alpha p^2 ) &\ll q^{-1/2+\eps}
    \sum_{\chi \mmod q} \bigg| \sum_{m \in \mathcal I'} \Lambda(m) \psi\chi(m) \mathrm e( \alpha m^2 ) \bigg| + x^{\theta-1/2}L \notag \\
    &\ll q^{-1/2+\eps}
    \sum_{\chi \mmod qr} \bigg| \sum_{m \in \mathcal I'} \Lambda(m) \chi(m) \mathrm e( \alpha m^2 ) \bigg| + x^{\theta-1/2}L,
  \end{align*}
  where $\mathcal I'$ is a subinterval of $\mathcal I$. Here, we have used that as $\chi$ runs through the characters modulo~$q$, the product $\psi\chi$ runs through a subset of the characters modulo $qr$. We now follow the rest of the proof of \cite[Theorem~1.1]{LiLvZh06} and find that the given exponential sum is
  \[
    \ll r^{1/2}(qrx)^{\eps}\left\{ x^{\theta-1/2}\Xi^{1/2} + x^{1/2}(qr)^{1/3}\Xi^{1/6} + x^{(7+5\theta)/15} + rx\Xi^{-1/2} \right\},
  \]
  where $\Xi = qrx^{2-2\theta}\big( 1 + x^{2\theta}|\alpha - a/q| \big)$. Under the hypotheses of the lemma, all the terms in this bound are $\ll rx^{\theta-\sigma+\eps}$.
\end{proof}

Suppose now that an arithmetic function $\lambda$ satisfies the following hypothesis:
\begin{enumerate}
  \item [(i*)] When $m \le 2x$, one can express $\lambda(m)$ as a linear combination of $O(L^c)$ convolutions of the form
  \[
    \sum_{\substack{ uv = m\\ u \sim U}} \xi_u\eta_v,
  \]
  where $|\xi_u| \le \tau(u)^c$, $U \le x^{\theta - 4\sigma}$, and either $\eta_v = \psi(v,z)$ with $z \le x^{2\theta-1-6\sigma}$, or $|\eta_v| \le \tau(v)^c$ and $U \ge x^{1-\theta+2\sigma}$.
\end{enumerate}
Let $\alpha \in \mathfrak m$ and let $a/q$ be a rational approximation to $\alpha$ of the form \eqref{eq4.2} with $Q_0$ given by \eqref{eq4.1}. If $q \ge x^{4\sigma}$, then hypothesis (i*) means that we can apply either Lemma \ref{lem4.1} or Lemma \ref{lem4.3} to show that
\[
  f(\alpha; \lambda) \ll x^{\theta-\sigma+\eps}.
\]
On the other hand, if $q \le x^{4\sigma}$, Lemma \ref{lem4.4} with a trivial character $\psi$ yields
\[
  f_0(\alpha) \ll x^{\theta - \sigma + \eps},
\]
unless we have $q \le x^{2\sigma-\eps}$ and $|q\alpha - a| \le x^{2\sigma - 2\theta + \eps}$. However, together, these two inequalities would place $\alpha$ in the set of major arcs $\mathfrak M$, contradicting our assumption that $\alpha \in \mathfrak m$. Therefore, if a sieve function $\lambda_i$ has the structure described in hypothesis (i*) above, then $\lambda_i$ satisfies hypothesis (i) in Section~\ref{sec2}. We shall use this observation in Section \ref{sec6} to construct the sieve functions $\lambda_1, \lambda_2$ and $\lambda_3$.

\subsection{Major arc estimates}

In this section, we collect estimates for averages of exponential sums $f(\beta;\lambda\chi)$ over sets of primitive characters $\chi$ and over small values of $\beta$. We are interested in arithmetic functions $\lambda$ having the structure described in the following hypothesis:
\begin{enumerate}
  \item [(ii)] When $m \le 2x$, one can express $\lambda(m)$ as a linear combination of $O(L^c)$ triple convolutions of the form
  \[
    \sum_{\substack{ uvw = m\\ u \sim U, v\sim V}} \xi_u\eta_v\zeta_w,
  \]
  where $|\xi_u| \le \tau(u)^c$, $|\eta_v| \le \tau(v)^c$, $\max(U,V) \le x^{11/20}$, and either $\zeta_w = 1$ for all $w$, or $|\zeta_w| \le \tau(w)^c$ and $UV \ge x^{27/35}$.
\end{enumerate}
The triple convolutions above may appear mysterious at first sight. They are chosen in accordance with \cite[Theorem 2.1]{ChKu06}, to ensure that that result can be applied at certain places in our proofs. The sieve functions $\lambda_i$ constructed in Section \ref{sec6} will all satisfy this hypothesis, and so does von Mangoldt's function $\Lambda$. The latter can be established by Heath-Brown's identity for $\Lambda$ along the lines of the proof of \cite[Theorem 1.1]{ChKu06}. Moreover, a variant of that argument using Linnik's identity instead of Heath-Brown's shows that the characteristic function of the primes $\varpi$ also satisfies hypothesis (ii).

\begin{lemma}\label{lem4.5}
  Let ${11}/{20} < \theta < 1$ and suppose that $P,Q$ satisfy
  \[
    PQ \le x^{1 + \theta}, \quad Q \ge x^{31/20+\eps}.
  \]
  Suppose also that $g$ is a positive integer and $\lambda$ is an arithmetic function satisfying hypothesis (ii) above. Then
  \begin{equation}\label{eq4.4}
    \sum_{r \le P} [g,r]^{-1+\varepsilon} \sumstar{\chi \mmod r}
    \bigg( \int_{-1/(rQ)}^{1/(rQ)} |f(\beta; \lambda\chi)|^2 \, d\beta \bigg)^{1/2}
    \ll g^{-1+2\eps}x^{(\theta-1)/2}L^c.
  \end{equation}
\end{lemma}

\begin{proof}
  Using a simple summation argument (see \cite[Lemma 1]{HaKu10}), we can deduce \eqref{eq4.4} from the bound
  \begin{equation}\label{eq4.5}
    \sum_{\substack{ r \sim R\\ d\mid r}} \; \sumstar{\chi \mmod r}
    \bigg( \int_{-1/(rQ)}^{1/(rQ)} |f(\beta, \lambda\chi)|^2 \, d\beta \bigg)^{1/2}
    \ll x^{(\theta-1)/2}L^c\big( 1 + d^{-1}RQ^{-1}x^{31/20} \big),
  \end{equation}
  where $1 \le R \le P$ and $1 \le d \le 2R$.

  We write $\Delta=(RQ)^{-1}$. By Gallagher's lemma (a variant of \cite[Lemma 5.4]{LiuZhan}), we have
  \[
    \int_{-\Delta}^{\Delta} |f(\beta, \lambda\chi)|^2 \, d\beta
    \ll \Delta^2 \int_{-\infty}^{\infty}\bigg|\sum_{m \in \mathcal I'(t)}\lambda(m)\chi(m) \bigg|^2 \, dt,
  \]
  where $\mathcal I'(t) = \mathcal I \cap [t^{1/2}, (t+(2\Delta)^{-1})^{1/2}]$. If we assume that $\Delta \ge x^{-1-\theta}$ (which follows from the hypothesis $PQ \le x^{1 + \theta}$), the sum over $m$ is non-empty for a set of values of $t$ having measure $\ll x^{1+\theta}$. Thus,
  \begin{align}\label{eq4.6}
    \int_{-\Delta}^{\Delta} |f(\beta, \lambda\chi)|^2 \, d\beta
    &\ll \Delta^2 x^{1+\theta} \bigg| \sum_{M \le m \le M+H} \lambda(m)\chi(m) \bigg|^2,
  \end{align}
  where $M \in \mathcal I$ and
  \[
    H \ll \min\big( x^{\theta}, (\Delta x)^{-1} \big) = (\Delta x)^{-1}.
  \]

  We note that without loss of generality, we may choose $M$ and $H$ so that $\| M \| = \| M+H \| = \frac 12$. Then, by Perron's formula \cite[Lemma 1.1]{LiuZhan},
  \[
    \sum_{M \le m \le M+H} \lambda(m)\chi(m)
    = \frac 1{2\pi\rmi} \int_{b-\rmi T_0}^{b+\rmi T_0} F(s, \chi) \frac{(M+H)^s-M^s}{s} \, ds + O(1),
  \]
  where $0<b<L^{-1}$, $T_0=x^{10}$, and
  \[
    F(s, \chi) = \sum_{m \sim 2x/3} \lambda(m)\chi(m)m^{-s}.
  \]
  Let $T_1 = \Delta x^{2}$. When $0<b<L^{-1}$, we have
  \[
    \frac {(M+H)^s-M^s}s \ll \frac 1{T_1 + |t|}.
  \]
  Thus, by letting $b \downarrow 0$, we deduce
  \begin{equation}\label{eq4.7}
    \sum_{M \le m \le M+H} \lambda(m)\chi(m) \ll
    \int_{-T_0}^{T_0} |F(\rmi t, \chi)| \, \frac {dt}{T_1 + |t|}  + 1.
  \end{equation}
  Let $\Sigma(R,d)$ denote the left side of \eqref{eq4.5}. Combining \eqref{eq4.6} and \eqref{eq4.7}, we get
  \begin{equation}\label{eq4.8}
    \Sigma(R,d) \ll x^{(1+\theta)/2}\Delta \bigg( LT^{-1} \sum_{\substack{ r\sim R\\ d\mid r}} \; \sumstar{\chi \mmod r}
    \int_{-T}^T |F(\rmi t, \chi)| \, dt + d^{-1}R^2 \bigg),
  \end{equation}
  for some $T$ with $T_1 \le T \le T_0$. Under hypothesis (ii), the above average can be estimated by \cite[Theorem~2.1]{ChKu06}. This yields
  \[
    \Sigma(R,d) \ll x^{(1+\theta)/2} \Delta L^c \big( xT_1^{-1} + d^{-1}R^2x^{11/20} \big),
  \]
  and \eqref{eq4.5} follows at once.
\end{proof}

The next lemma can be proved similarly to \cite[Lemma 3.2]{Lv05} with \cite[Lemma~5.1]{Lv05} replaced by \cite[Theorem~2.1]{ChKu06} and Heath-Brown's identity by hypothesis (ii).

\begin{lemma}\label{lem4.6}
  Let $7/10 < \theta < 1$ and suppose that $P,Q$ satisfy
  \[
    P \le x^{\theta-11/20-\eps}, \quad PQ^{-1} \le x^{2\theta - 31/10 - \eps}, \quad PQ \le x^{1+\theta}.
  \]
  Suppose also that $g$ is a positive integer and $\lambda$ is an arithmetic function satisfying hypothesis (ii) above. Then
  \begin{equation}\label{eq4.9}
    \sum_{r \le P} [g,r]^{-1+\varepsilon} \sumstar{\chi \mmod r}
    \max_{|\beta| \le 1/(rQ)} |f(\beta; \lambda\chi)| \ll g^{-1+2\eps}x^{\theta}L^c.
  \end{equation}
  Furthermore, for any given $A > 0$, there is a $B=B(A) > 0$ such that
  \begin{equation}\label{eq4.10}
    \sum_{L^B \le r \le P} r^{-1+\varepsilon} \sumstar{\chi \mmod r}
    \max_{|\beta| \le 1/(rQ)} |f(\beta; \lambda\chi)| \ll x^{\theta}L^{-A}.
  \end{equation}
\end{lemma}

The final lemma in this section extends the range of $r$ in \eqref{eq4.10} below $L^B$, though only in the special case $\lambda = \varpi$. We should point, however, that the restriction to $\varpi$ is merely for convenience. In principle, the result can be extended to include exponential sums with sieve weights, though that would require a more technical proof. Since the more general case is not needed in this paper, we select simplicity over generality here.

\begin{lemma}\label{lem4.7}
  Let ${19}/{24} < \theta < 1$ and suppose that $P,Q$ satisfy
  \[
    P \le x^{\theta-11/20-\eps}, \quad PQ^{-1} \le x^{2\theta - 31/10 - \eps}, \quad PQ \le x^{1+\theta}.
  \]
  Then, for any given $A > 0$,
  \begin{equation}\label{eq4.11}
    \sum_{r \le P} r^{-1+\varepsilon} \sumstar{\chi \mmod r}
    \max_{|\beta| \le 1/(rQ)} |f(\beta; \varpi\chi) - \delta_{\chi}v(\beta)| \ll x^{\theta}L^{-A},
  \end{equation}
  where
  \begin{equation}\label{eq4.12}
    v(\beta) = \sum_{m \in \mathcal I} \frac {\mathrm e(\beta m^2)}{\log m}.
  \end{equation}
\end{lemma}

\begin{proof}
  By the second part of Lemma \ref{lem4.6}, it suffices to show that
  \begin{equation}\label{eq4.13}
    \max_{|\beta| \le 1/Q} |f(\beta; \varpi\chi) - \delta_{\chi}v(\beta)| \ll x^{\theta}L^{-A-2B}
  \end{equation}
  for all characters $\chi$ with moduli $q \le L^B$, where $B = B(A)$ is the number appearing in \eqref{eq4.10}. When $|\beta| \ge x^{-2\theta + \eps}$, Lemma \ref{lem4.4} with $\sigma = \eps/2$, $\eps = \eps^2$, $a = 0$, $q = 1$, and $\psi = \chi$ yields
  \[
    f(\beta; \varpi\chi) \ll x^{\theta - \eps/3}.
  \]
  Thus, we may assume that $Q \ge x^{2\theta - \eps}$ in \eqref{eq4.13}. (In the case $r = 1$, we also need to estimate the main term $v(\beta)$ for $|\beta| \ge x^{-2\theta + \eps}$; that can be done using \eqref{eq5.20} below.)

  We now argue similarly to the proof of \cite[Lemma 3.3]{Lv05}. By partial summation and the arguments in \cite{Lv05}, we obtain
  \[
    f(\beta; \varpi\chi) - \delta_{\chi}v(\beta) \ll x^{\theta} \sum_{|\im(\rho)| \le T} x^{\re(\rho) - 1} + x^{\theta - \eps},
  \]
  where $T = x^{2 - 2\theta + 3\eps}$ and the summation is over the non-trivial zeros of the Dirichlet $L$-function $L(s, \chi)$. As is customary, let $N(\chi; \alpha, T)$ denote the number of zeros of $L(s,\chi)$ in the region
  \[
    |\im(s)| \le T, \quad \re(s) \ge \alpha.
  \]
  By the zero-free region for Dirichlet $L$-functions \cite[Theorem 1.10]{LiuZhan} and Siegel's theorem on exceptional zeros \cite[Theorem 1.12]{LiuZhan}, there exists a constant $c_0 > 0$ such that $N(\chi; \alpha, T) = 0$ when $\alpha \ge 1 - \eta(T)$, where $\eta(T) = c_0(\log T)^{-4/5}$. Therefore, by Huxley's zero-density theorem,
  \begin{align*}
    \sum_{|\im(\rho)| \le T} x^{\re(\rho) - 1} &= -\int_0^{1-\eta(T)} x^{\alpha-1} \, dN(\chi; \alpha, T) \\
    &\ll L\int_0^{1 - \eta(T)} x^{\alpha-1}N(\chi; \alpha, T) \, d\alpha + x^{-1}TL\\
    &\ll L\big( x^{-1}T^{12/5} \big)^{\eta(T)} + x^{-\eps} \ll \exp\big( -c_1L^{1/5} \big),
  \end{align*}
  provided that $\eps$ is chosen sufficiently small.
\end{proof}

\section{Proof of Theorem \ref{thm2}: The major arcs}
\label{sec5}

In this section, we establish \eqref{eq2.9} and describe the numbers $K_n$ and $C$ appearing there. First, we need to make some assumptions about the structure and asymptotic behavior of the functions $\lambda_1, \lambda_2$ and $\lambda_3$. We suppose that they satisfy hypothesis (ii) in Section \ref{sec4} and the following three additional hypotheses:
\begin{enumerate}
  \item [(iii)] There is a $z \ge x^{\sigma}$ such that $\lambda(m) = 0$ whenever $m$ is divisible by a prime $p < z$.
  \item [(iv)] Let $A, B > 0$ be fixed, let $\chi$ be a non-principal character modulo $q$, $q \le L^B$, and let $\mathcal I'$ be a subinterval of $\mathcal I$. Then
  \[
    \sum_{m \in \mathcal I'} \lambda(m)\chi(m) \ll x^{\theta}L^{-A}.
  \]
  \item [(v)] Let $A > 0$ be fixed and let $\mathcal I'$ be a subinterval of $\mathcal I$. There exists a constant $\kappa > 0$ such that
  \[
    \sum_{m \in \mathcal I'} \lambda(m) = \kappa|\mathcal I'|L^{-1} + O\big( x^{\theta}L^{-A} \big).
  \]
\end{enumerate}
To fully appreciate hypotheses (iv) and (v), one may consider what they say in the special case when $\lambda = \lambda_0$, the characteristic function of the primes. In that case, (iv) is a short-interval version of the Siegel--Walfisz theorem, and (v), with $\kappa_0 = 1$, is a short-interval version of the Prime Number Theorem.

Under hypotheses (ii)--(v), the evaluation of the right side of \eqref{eq2.9} is similar to (although somewhat more technical than) the analysis of the major arcs in earlier work on the problem. Let $\kappa_i$ denote the constant $\kappa$ appearing in hypothesis (v) for $\lambda_i$, and define the singular integral $\mathfrak J(n)$ and the singular series $\mathfrak S(n)$ by
\begin{gather*}
  \mathfrak J(n) = \int_0^1 v(\beta)^4\mathrm e(-\beta n) \, d\beta, \quad
  \mathfrak S(n) = \sum_{q = 1}^{\infty} A(n; q), \\
  A(n;q) = \phi(q)^{-4}\sum_{\substack{1 \le a \le q\\(a, q)=1}} S(q,a)^4\mathrm e(-an/q), \notag
\end{gather*}
where $v(\beta)$ is defined by \eqref{eq4.12} and
\[
  S(q,a) = \sum_{\substack{1 \le h \le q\\ (h,q) = 1}} \mathrm e(ah^2/q).
\]
We establish the following result.

\begin{proposition}\label{prop1}
  Let $\theta$ and $\sigma$ be fixed real numbers with
  \begin{equation}\label{eq5.1}
    19/24 < \theta < 1 \quad \text{and} \quad 0 < \sigma \le \theta - 31/40,
  \end{equation}
  and let the major arcs $\mathfrak M$ be given by \eqref{eq2.4} and \eqref{eq2.5}. Suppose further that the arithmetic functions $\lambda_j$ and $\lambda_k$, $0 \le j,k \le 3$, satisfy hypotheses (ii)--(v) and the exponential sums $g_j(\alpha)$ and $g_k(\alpha)$ are defined by \eqref{eq5.2} below. Then
  \[
    \int_{\mathfrak M} f_0(\alpha)^2g_j(\alpha)g_k(\alpha)\mathrm e(-\alpha n) \, d\alpha
    = \mathfrak S(n)\mathfrak I(n) \big( \kappa_j\kappa_k + O \big( L^{-1} \big) \big).
  \]
\end{proposition}

Using standard major arc techniques (see Liu and Zhan \cite[Lemma 8.3]{LiuZhan} and a variant of \cite[Lemma 6.4]{LiuZhan}), we find that when $n \in \mathcal H_4$ and $|n - X| \le X^{(\theta+1)/2}$,
\begin{equation}\label{eq5.1a}
  x^{3\theta-1}L^{-4} \ll \mathfrak S(n)\mathfrak J(n) \ll x^{3\theta-1}L^{-3}.
\end{equation}
Thus, we can use Proposition \ref{prop1} to deduce \eqref{eq2.9} with
\[
  K_n = x^{1-3\theta}L^{4}\mathfrak S(n)\mathfrak J(n)
  \quad \text{and} \quad C = \kappa_1 - \kappa_2\kappa_3.
\]
We now proceed to prove the proposition.

\subsection*{Proof of Proposition \ref{prop1}}

We commence our analysis of the major arcs by defining the exponential sums $g_i(\alpha)$ appearing in the statement (and in \eqref{eq2.9} \emph{via} $G(\alpha)$). Define the function $\omega$ on $\mathcal I \times \mathfrak M$ by
\[
  \omega(m, \alpha) = \begin{cases}
    0 & \text{if } \alpha \in \mathfrak M(q, a) \text{ and } (m, q) \ge x^{\sigma}, \\
    1 & \text{otherwise}.
  \end{cases}
\]
For $\alpha \in \mathfrak M$, we set
\begin{equation}\label{eq5.2}
  g_i(\alpha) = \sum_{m \in \mathcal{I}} \lambda_i(m)\omega(m, \alpha) \mathrm e(\alpha m^2),
\end{equation}
and then define $h_i(\alpha)$ by \eqref{eq2.7}. We note that it is convenient to include $i=0$ in this definition, even though in that case we have $g_0(\alpha) = f_0(\alpha)$. Let $\alpha$ be on a major arc $\mathfrak M(q, a)$. When $\lambda_i$ satisfies hypothesis (iii), the sum $g_i(\alpha)$ is supported on integers $m$ with $(m,q) = 1$. Furthermore, since an integer $q \le x^{2\sigma-\eps}$ can have at most one prime divisor $p \ge z$, the sum $h_i(\alpha)$ is supported on integers divisible by $p$ (if $p$ exists) and satisfies
\[
  h_i(\alpha) \ll \sum_{ \substack{ m \in \mathcal I\\ p \mid m}} |\lambda_i(m)| \ll x^{\theta - \sigma}.
\]
This verifies \eqref{eq2.8}. We remark that although the prime $p$ in this bound depends on the major arc $\mathfrak M(q, a)$, the bound itself is uniform.

When $0 \le i \le 3$, we define the function $f_i^*(\alpha)$ on $\mathfrak M$ by setting
\[
  f_i^*(\alpha) = \kappa_i\phi(q)^{-1}S(q,a) v(\alpha - a/q)
  \quad \text{if } \alpha \in \mathfrak M(q, a).
\]
This function is the major arc approximation to $f_i(\alpha)$ suggested by hypotheses (iv) and (v). We now proceed to show that we can replace the exponential sums $g_i(\alpha)$ in \eqref{eq2.9} by the respective $f_i^*(\alpha)$. We shall show that
\begin{align}\label{eq5.3}
  \int_{\mathfrak{M}} \big(f_0(\alpha)^2g_j(\alpha)g_k(\alpha) - f_0^*(\alpha)^2f_j^*(\alpha)f_k^*(\alpha)\big)
  \mathrm e(-\alpha n) \, d\alpha \ll x^{3\theta-1}L^{-A}
\end{align}
for any fixed $A>0$.

Let $\alpha \in \mathfrak M(q, a)$. Then, similarly to \cite[eq. (4.1)]{HaKu06} (it is here that we make use of the weights $\omega(m,\alpha)$), we have
\begin{equation}\label{eq5.4}
  g_i(\alpha) = \phi(q)^{-1}\sum_{\chi \mmod q} S(\chi,a)f(\alpha - a/q; \lambda_i\chi),
\end{equation}
where
\[
  S(\chi,a) = \sum_{h=1}^{q} \bar\chi(h) \mathrm e( ah^2/q ).
\]
Hence,
\begin{equation}\label{eq5.5}
  f_0(\alpha) = f_0^*(\alpha) + \Delta_0(\alpha), \quad g_i(\alpha) = f_i^*(\alpha) + \Delta_i(\alpha),
\end{equation}
with
\begin{gather*}
  \Delta_i(\alpha) = \phi(q)^{-1}\sum_{\chi \mmod q}
  S(\chi,a) W_i(\alpha - a/q, \chi), \\
  W_i(\beta, \chi) = f(\beta; \lambda_i\chi - \kappa_i\rho_{\chi}).
\end{gather*}
Here, $\rho_\chi(m) = (\log m)^{-1}$ or $0$ according as the character $\chi$ is principal or not. Using \eqref{eq5.5}, we can express the integral in \eqref{eq5.3} as the linear combination of eleven integrals of the form
\begin{equation}\label{eq5.6}
  \int_{\mathfrak M} f_0^*(\alpha)^{2-a}\Delta_0(\alpha)^a f_j^*(\alpha)^{1-b}\Delta_j(\alpha)^b
  f_k^*(\alpha)^{1-c}\Delta_k(\alpha)^c \mathrm e(-\alpha n) \, d\alpha,
\end{equation}
where $a \in \{ 0, 1, 2 \}$, $b,c \in \{ 0,1 \}$, and $a+b+c > 0$. The estimation of all those integrals follows the same pattern, so we shall focus on the most troublesome among them---namely,
\begin{equation}\label{eq5.7}
  \int_{\mathfrak M} \Delta_0(\alpha)^2\Delta_j(\alpha)\Delta_k(\alpha) \mathrm e(-\alpha n) \, d\alpha.
\end{equation}

We can rewrite \eqref{eq5.7} as the multiple sum
\begin{equation}\label{eq5.8}
  \sum_{q\le P}\sum_{\chi_1\mmod q}\cdots\sum_{\chi_4\mmod q}B(q;\chi_1,\dots,\chi_4)J(q;\chi_1,\dots,\chi_4),
\end{equation}
where
\[
  B(q;\chi_1,\dots,\chi_4)= \phi(q)^{-4} \sum_{\substack{ 1 \le a \le q\\ (a,q)=1}}
  S(\chi_1,a)\cdots S(\chi_4,a) \mathrm e( -an/q )
\]
and
\[
  J(q;\chi_1,\dots,\chi_4) = \int_{-1/(qQ)}^{1/(qQ)}
  W_0(\beta, \chi_1)W_0(\beta, \chi_2)W_j(\beta, \chi_3)W_k(\beta, \chi_4) \mathrm e(-\beta n) \, d\beta.
\]

We first reduce \eqref{eq5.8} to a sum over primitive characters. In general, if $\chi$ modulo $q$, $q\le P$, is induced by a primitive character $\chi^*$ modulo $r$, we have
\begin{equation}\label{eq5.9}
  W_0(\beta, \chi) = W_0(\beta, \chi^*)
\end{equation}
and, by hypothesis (iii),
\begin{equation}\label{eq5.10}
  W_i(\beta, \chi) = W_i(\beta, \chi^*) + O\big( x^{\theta}z^{-1} \big)
\end{equation}
for $i \ge 1$. The error term in \eqref{eq5.10} accounts for the terms in $f(\beta; \lambda_i\chi)$ with $m$ satisfying
\[
  m \in \mathcal{I}, \quad (m,q)>1, \quad (m,r)=1, \quad \lambda_j(m)\neq 0.
\]
In particular, that error term is superfluous when $r > Pz^{-1}$, as the set of such $m$ is then empty. Thus, we can strengthen \eqref{eq5.10} to
\begin{equation}\label{eq5.11}
  W_i(\beta, \chi) = W_i(\beta, \chi^*) + O( \Psi(r) ),
\end{equation}
where
\[
  \Psi(r) = \begin{cases}
    x^{\theta}z^{-1} & \text{if } r \le Pz^{-1}, \\
    0                & \text{if } r > Pz^{-1}.
  \end{cases}
\]
Given a character $\chi$ modulo $r$, we define
\begin{gather*}
  V_i(\chi) = \max_{|\beta| \le 1/(rQ)} |W_i(\beta, \chi)| ,\\
  W_i(\chi) = \bigg( \int_{-1/(rQ)}^{1/(rQ)}|W_i(\beta, \chi)|^2 \, d\beta \bigg)^{1/2}.
\end{gather*}
Let $\chi_i^*$ modulo $r_i$, $r_i \mid q$, be the primitive character inducing $\chi_i$ and set $q_0 = [r_1, \dots, r_4]$. By \eqref{eq5.9} and \eqref{eq5.11},
\begin{align} \label{eq5.12}
  J(q; \chi_1,\dots,\chi_4) \ll {}& V_0(\chi_1^*) V_0(\chi_2^*) W_j(\chi_3^*) W_k(\chi_4^*) \notag\\
  &+ \Psi(r_3)V_0(\chi_1^*) W_0(\chi_2^*) W_k(\chi_4^*) \notag\\
  &+ \Psi(r_4)V_0(\chi_1^*) W_0(\chi_2^*) W_j(\chi_3^*) \notag\\
  &+ \Psi(r_3)\Psi(r_4)W_0(\chi_1^*) W_0(\chi_2^*).
\end{align}
Let $J_i^*(\chi_1^*, \dots, \chi_4^*)$, $1 \le i \le 4$, denote the $i$th term on the right side of \eqref{eq5.12}. The sum \eqref{eq5.8} does not exceed
\begin{equation}\label{eq5.13}
  \sum_{1 \le i \le 4} \sumstar{r_1,\chi_1} \cdots \sumstar{r_4,\chi_4}
  J_i^*(\chi_1, \dots, \chi_4) B_0(\chi_1,\dots,\chi_4),
\end{equation}
with
\[
  B_0(\chi_1,\dots,\chi_4) = \sum_{\substack{q\le P\\ q_0 \mid q}} |B(q; \chi_1,\dots,\chi_4)|.
\]
By \cite[Lemma 6.3]{LiuZhan},
\[
  B_0(\chi_1,\dots,\chi_4) \ll q_0^{-1+\eps}L^c.
\]
Hence, the sum \eqref{eq5.13} is
\begin{equation}\label{eq5.14}
  \ll L^c \sum_{1 \le i \le 4} \sumstar{r_1,\chi_1} \cdots \sumstar{r_4,\chi_4}
  q_0^{-1+\eps} J_i^*(\chi_1, \dots, \chi_4).
\end{equation}

Before we proceed to estimate the last sum, we stop to remark that inequalities \eqref{eq5.1} ensure that Lemmas \ref{lem4.5}--\ref{lem4.7} are applicable with $P$ and $Q$ given by \eqref{eq2.4}. Indeed, altogether the three lemmas require that ${19}/{24} < \theta < 1$ and that $P$ and $Q$ satisfy the following inequalities:
\[
  Q \ge x^{31/20+\eps}, \quad
  P \le \min\big( x^{\theta-11/20-\eps}, x^{2\theta-31/10-\eps}Q, x^{1+\theta}Q^{-1} \big).
\]
We first note that when $Q \ge x^{31/20+\eps}$, we have $x^{1+\theta}Q^{-1} \le x^{\theta-11/20-\eps}$; hence, the condition $P \le x^{\theta-11/20-\eps}$ is superfluous. Further, when $P$ and $Q$ are chosen according to \eqref{eq2.4}, the condition $PQ \le x^{1+\theta}$ follows from the assumption that $\theta < 1$. The remaining two constraints,
\[
  Q \ge x^{31/20+\eps} \quad \text{and} \quad P \le x^{2\theta-31/10-\eps}Q,
\]
are satisfied if $\theta - \sigma \ge 31/40$.

Next we estimate \eqref{eq5.14} by a standard iterative procedure. We write $\Sigma_i$ for the $i$th term in \eqref{eq5.14} and focus on $\Sigma_1$. We have
\[
  \sumstar{r_4,\chi_4} [g,r_4]^{-1+\eps} W_k(\chi_4) \ll \Sigma(g; \lambda_k) + g^{-1+\eps}I_k,
\]
where $\Sigma(g; \lambda)$ is the sum appearing on the left side of \eqref{eq4.4} and
\begin{align*}
  I_k^2 = \int_{-1/Q}^{1/Q} |v(\beta)|^2 \, d\beta
  &\ll \sum_{m_1,m_2 \in \mathcal I} \frac {L^{-2}}{Q + |m_1^2 - m_2^2|}\\
  &\ll x^{\theta-1} + x^{\theta}Q^{-1} \ll x^{\theta-1}.
\end{align*}
Using this bound for $I_k$ and Lemma \ref{lem4.5}, we conclude that
\[
  \sumstar{r_4,\chi_4} [g,r_4]^{-1+\eps} W_k(\chi_4) \ll g^{-1+2\eps}x^{(\theta-1)/2}L^c.
\]
By this inequality and the analogous bound for the sum over $r_3,\chi_3$, we see that
\begin{equation}\label{eq5.15}
  \Sigma_1 \ll x^{\theta - 1}L^c \sumstar{r_1,\chi_1} V_0(\chi_1) \sumstar{r_2,\chi_2} [r_1, r_2]^{-1+3\eps} V_0(\chi_2).
\end{equation}
We now use Lemma \ref{lem4.6} to estimate the inner sum in \eqref{eq5.15} and obtain
\begin{equation}\label{eq5.16}
  \Sigma_1 \ll x^{2\theta - 1}L^c \sumstar{r_1,\chi_1} r_1^{-1+4\eps} V_0(\chi_1).
\end{equation}
Finally, we apply Lemma \ref{lem4.7} to the last sum and conclude that $\Sigma_1 \ll x^{3\theta-1}L^{-A}$ for any fixed $A > 0$.

The estimation of the sums $\Sigma_i$, with $i \ge 2$, is similar and, in fact, simpler than that of $\Sigma_1$. We demonstrate the necessary changes in the case of $\Sigma_2$. We follow the above argument until we reach \eqref{eq5.16} which is now replaced by
\[
  \Sigma_2 \ll x^{2\theta - 1}L^c \sumstar{r_3,\chi_3} r_3^{-1+4\eps}\Psi(r_3)
  \ll x^{3\theta-1}z^{-2}P^{1 + 4\eps}L^c.
\]
This bound is $\ll x^{3\theta-1-\eps}$ provided that $P \le z^2x^{-4\eps}$, for example. Similar arguments show that $\Sigma_3$ and $\Sigma_4$ are also $\ll x^{3\theta-1-\eps}$. Therefore, the integral \eqref{eq5.7} is $O(x^{3\theta-1}L^{-A})$ for any fixed $A > 0$.

We can argue similarly to estimate other integrals of the form \eqref{eq5.6} which include at least one factor $\Delta_0(\alpha)$ (i.e., where $a > 0$). When no such factor is present, we need to adjust the above argument slightly to make use of hypotheses (iv) and (v) about the sieve functions $\lambda_i$. Let us consider one such integral---say,
\begin{equation}\label{eq5.17}
  \int_{\mathfrak{M}} f_0^*(\alpha)^2\Delta_j(\alpha)\Delta_k(\alpha) \mathrm e(-\alpha n) \, d\alpha.
\end{equation}
This integral equals
\begin{equation}\label{eq5.18}
  \sum_{q\le P}\sum_{\chi_3\mmod q}\sum_{\chi_4\mmod q}B(q;\chi_0,\chi_0,\chi_3,\chi_4)\tilde J(q;\chi_3,\chi_4),
\end{equation}
where $\chi_0$ denotes the principal character modulo $q$ (hence, $S(\chi_0,a) = S(q,a)$) and
\[
  \tilde J(q;\chi_3,\chi_4) = \int_{-1/(qQ)}^{1/(qQ)}
  v(\beta)^2W_j(\beta, \chi_3)W_k(\beta, \chi_4) \mathrm e(-\beta n) \, d\beta.
\]
Passing to primitive characters, we obtain a variant of \eqref{eq5.14} for the sum \eqref{eq5.18}. The terms corresponding to $\Sigma_2, \Sigma_3$ and $\Sigma_4$ in \eqref{eq5.14} can be estimated as before, so we concentrate on the remaining sum,
\begin{equation}\label{eq5.19}
  \sumstar{r_3, \chi_3}\sumstar{r_4,\chi_4} \tilde q_0^{-1+\eps}
  \int_{-1/(\tilde q_0Q)}^{1/(\tilde q_0Q)} |v(\beta)^2W_j(\beta, \chi_3)W_k(\beta, \chi_4)| \, d\beta.
\end{equation}
where $\tilde q_0 = [r_3,r_4]$. Using \cite[Lemma 1.19]{LiuZhan}, we get
\begin{equation}\label{eq5.20}
  v(\beta) \ll \frac {x^{\theta}L^{-1}}{1 + x^{1+\theta}\|\beta\|}.
\end{equation}
Thus, \eqref{eq5.19} is bounded above by
\begin{equation}\label{eq5.21}
  x^{2\theta} \sumstar{r_3, \chi_3}\sumstar{r_4,\chi_4} \tilde q_0^{-1+\eps}
  \int_{-1/(\tilde q_0Q)}^{1/(\tilde q_0Q)} |W_j(\beta, \chi_3)W_k(\beta, \chi_4)| \, \frac {d\beta}{(1 + x^{1+\theta}|\beta|)^2}.
\end{equation}

We now split the last sum in two. First, we consider the terms in \eqref{eq5.21} with $\max(r_3, r_4) \ge L^{B_1}$, where $B_1 = B_1(A) > 0$ is to be chosen shortly. Without loss of generality, we may assume that $r_3 \ge r_4$. In this case, we note that the integral over $\beta$ in \eqref{eq5.21} is $\ll x^{-1-\theta}V_j(\chi_3)V_k(\chi_4)$. We apply Lemma \ref{lem4.6} first to the sum over $r_4,\chi_4$ and then to the one over $r_3,\chi_3$. We conclude that the contribution to \eqref{eq5.21} from terms with $\max(r_3, r_4) \ge L^{B_1}$ is $O(x^{3\theta-1}L^{-A})$ for any fixed $A > 0$, provided that $B_1$ is sufficiently large. To be more precise, it suffices to choose $B_1 = B(A + c)$, where $c > 0$ is an absolute constant and $B(A)$ is the function of $A$ appearing in the second part of Lemma \ref{lem4.6}.

We now turn to the terms in \eqref{eq5.21} with $\max(r_3, r_4) \le L^{B_1}$. The contribution to \eqref{eq5.21} from such moduli does not exceed
\begin{equation}\label{eq5.22}
  x^{2\theta}L^{3B_1} \int_{-1/Q}^{1/Q} |W_j(\beta, \chi_3)W_k(\beta, \chi_4)| \, \frac {d\beta}{(1 + x^{1+\theta}|\beta|)^2},
\end{equation}
for some characters $\chi_3$ and $\chi_4$ with moduli $\le L^{B_1}$. Let $B_2 = 3B_1 + A$ and $Q_0 = x^{1+\theta}L^{-B_2}$. The contribution to \eqref{eq5.22} from $\beta$ with $|\beta| \ge Q_0^{-1}$ can be estimated trivially as
\[
  \ll x^{4\theta}L^{3B_1} \int_{1/Q_0}^{\infty} \frac {d\beta}{(1 + x^{1+\theta}|\beta|)^2}
  \ll x^{3\theta-1}L^{-A}.
\]
Finally, we estimate the contribution to \eqref{eq5.22} from $\beta$ with $|\beta| \le Q_0^{-1}$. By partial summation,
\[
  W_j(\beta, \chi) \ll \big( 1+x^{1+\theta}|\beta| \big)
  \sup_{\mathcal I'} \bigg| \sum_{m \in \mathcal I'} (\lambda_j(m) - \kappa_j\rho_\chi(m)) \chi(m) \bigg|,
\]
where the supremum is over all the subintervals $\mathcal I'$ of $\mathcal I$. When the character $\chi$ has a modulus $r \le L^{B_1}$, the above sum can be estimated by hypothesis (iv) or (v) with $A$ replaced by $3B_1 + A + 1$. Thus, the contribution to \eqref{eq5.22} from $|\beta| \le Q_0^{-1}$ is
\[
  \ll x^{4\theta}L^{-A-1} \int_{-1/Q_0}^{1/Q_0} \frac {d\beta}{1 + x^{1+\theta}|\beta|} \ll x^{3\theta-1}L^{-A}.
\]
This concludes the estimation of \eqref{eq5.17}. Thus, we have now established \eqref{eq5.3} for any fixed $A > 0$.

Finally, we evaluate
\begin{equation}\label{eq5.23}
  \int_{\mathfrak{M}} f_0^*(\alpha)^2f_j^*(\alpha)f_k^*(\alpha) \mathrm e(-\alpha n) \, d\alpha
  = \kappa_j\kappa_k \sum_{q \le P}A(n;q) \mathfrak J(n;1/(qQ)),
\end{equation}
where
\[
  \mathfrak J(n;\Delta) = \int_{-\Delta}^{\Delta} v(\beta)^4\mathrm e(-\beta n) \, d\beta.
\]
Using \eqref{eq5.20} and the bound
\[
  A(n;q) \ll (n,q)^{1/2}q^{-3/2+\eps},
\]
we can first replace $\mathfrak J(n;1/(qQ))$ by $\mathfrak J(n) = \mathfrak J(n;\frac 12)$ and then extend the summation over $q$ to~$\infty$. Since $PQ \le x^{1+\theta}$, this yields
\begin{equation}\label{eq5.24}
  \sum_{q \le P}A(n;q) \mathfrak J(n;1/(qQ)) = \mathfrak S(n)\mathfrak J(n) + O\big( x^{3\theta - 1}P^{-1/2+\eps} \big).
\end{equation}
The proposition follows from \eqref{eq5.3}, \eqref{eq5.23} and \eqref{eq5.24}. \qed

\section{Proof of Theorem \ref{thm2}: The sieve}
\label{sec6}

In this section, we complete the proof of Theorem \ref{thm2}. Let $8/9 < \theta < 1$ and set $\sigma = (2\theta - 1)/7$.
We note that we then have
\[
  1 - \theta < \sigma < \min( (3\theta - 2)/6, \theta - 31/40 ).
\]
We construct sieve weights $\lambda_1, \lambda_2$ and $\lambda_3$ that satisfy \eqref{eq2.1} and the various hypotheses required by the analysis in Sections \ref{sec2}, \ref{sec3} and \ref{sec5}. Since the construction is driven by our need to have $\lambda_1$ and $\lambda_3$ satisfy hypothesis (i$^*$) in Section \ref{sec4}, it is convenient to set
\[
  U = x^{1 - \theta + 2\sigma}, \quad V = x^{\theta - 4\sigma}, \quad z = x^{2\theta - 1 - 6\sigma}.
\]
When $\theta$ and $\sigma$ are as above, these quantities satisfy
\[
  x^{1/9} < x^{\sigma} = z < z^2 < U < z^3 < V = Uz < x^{4/9}.
\]
Our construction uses repeatedly Buchstab's identity from sieve theory, which can be stated as
\begin{equation}\label{eq6.1}
  \psi(m,z_1)=\psi(m,z_2)-\sum_{z_2 \le p < z_1}\psi(m/p,p) \qquad (2\le z_2<z_1),
\end{equation}
in view of our definition of $\psi(m, z)$ (recall \eqref{eq1.8} and the convention on non-integer $m$).

We start from the simple observation that when $m \in \mathcal I$, we have $\varpi(m) = \psi(m, x_1^{1/2})$, $x_1 = x + x^{\theta}$. Thus, Buchstab's identity yields
\begin{align}\label{eq6.2}
  \varpi(m) &= \psi(m,z)
  - \bigg\{  \sum_{z \le p < U} + \sum_{U \le p \le V} + \sum_{V<p<x_1^{1/2}} \bigg\} \, \psi(m/p,p) \notag\\
  &= \gamma_1(m) - \gamma_2(m) - \gamma_3(m) - \gamma_4(m), \quad\text{say}.
\end{align}
We remark that $\gamma_1$ and $\gamma_3$ satisfy hypothesis (i$^*$). We decompose $\gamma_2$ further using Buchstab's identity again. We have
\begin{align}\label{eq6.3}
  \gamma_2(m) &= \sum_{z \le p< U} \psi(m/p,z) - \sum_{z \le p_2 < p_1 < U}
  \bigg\{\sum_{p_1p_2<U} + \sum_{U\le p_1p_2\le V} + \sum_{p_1p_2>V} \bigg\} \, \psi(m/(p_1p_2),p_2) \notag\\
  &= \gamma_5(m) - \gamma_6(m) - \gamma_7(m) - \gamma_8(m), \quad\text{say}.
\end{align}
Here, $\gamma_5$ and $\gamma_7$ satisfy hypothesis (i$^*$). We decompose $\gamma_6$ further as
\begin{align}\label{eq6.4}
  \gamma_6(m) &= \sum_{z \le p_2 < p_1 < p_1p_2<U} \psi(m/(p_1p_2),z) \notag\\
  &\phantom{={}}- \sum_{z\le p_2<p_1<p_1p_2<U} \bigg\{ \sum_{ \substack{ z\le p_3<p_2 \\ p_1p_2p_3\le V}}
  + \sum_{ \substack{ z\le p_3<p_2 \\ p_1p_2p_3 > V}} \bigg\} \psi(m/(p_1p_2p_3), p_3) \notag\\
  &= \gamma_{9}(m) - \gamma_{10}(m) - \gamma_{11}(m),\quad\text{say}.
\end{align}
Note that the condition $p_1p_2p_3 \ge U$ is implicit in $\gamma_{10}$ because $z^3 > U$. Finally, we combine \eqref{eq6.2}--\eqref{eq6.4} and deduce that
\begin{equation}\label{eq6.5}
  \varpi(m) = \lambda_1(m) - \lambda_2(m) + \gamma_8(m) \ge \lambda_1(m) - \lambda_2(m),
\end{equation}
where
\begin{align*}
  \lambda_1(m) &= \gamma_1(m) - \gamma_3(m) - \gamma_5(m) + \gamma_7(m) + \gamma_{9}(m) - \gamma_{10}(m),\\
  \lambda_2(m) &= \gamma_4(m) + \gamma_{11}(m) \ge 0.
\end{align*}
We remark that $\lambda_1$ is the sum of those $\gamma_i$'s which satisfy hypothesis (i$^*$), while $-\lambda_2$ collects the negative among the remaining terms in the decomposition of $\varpi$.

Next, we construct $\lambda_3$. Similarly to \eqref{eq6.2}--\eqref{eq6.4}, using three Buchstab decompositions, we have
\begin{align*}
  \varpi(m) &= \gamma_1(m) - \gamma_3(m) - \gamma_4(m) -
  \sum_{V^{1/2}< p < U} \psi(m/p,p) - \sum_{z \le p \le V^{1/2}}\psi(m/p,z) \\
  &\phantom{={}}+  \sum_{z\le p_2<p_1\le V^{1/2}} \psi(m/p_1p_2,z) -
  \sum_{z\le p_3<p_2<p_1\le V^{1/2}} \psi(m/p_1p_2p_3,p_3) \\
  &= \gamma_1(m) - \gamma_3(m) - \gamma_4(m) - \gamma_5^*(m) - \gamma_6^*(m)
  + \gamma_7^*(m) - \gamma_8^*(m),\quad \text{say.}
\end{align*}
We note that $\gamma_1, \gamma_3, \gamma_6^*$ and $\gamma_7^*$ satisfy hypothesis (i$^*$). Let $\gamma_9^*(m)$ denote the portion of $\gamma_8^*(m)$ where either $p_1p_2 \ge U$ or $p_1p_2p_3 \le V$. Since $p_1p_2 \le V$ and $p_1p_2p_3 > z^3 > U$, $\gamma_9^*$ too satisfies hypothesis~(i$^*$). We now set
\begin{equation}\label{eq6.6}
  \lambda_3(m) = \gamma_1(m) - \gamma_3(m)- \gamma_6^*(m)
  + \gamma_7^*(m) - \gamma_9^*(m).
\end{equation}
Then $\lambda_3(m) \ge \varpi(m)$ and it also satisfies hypothesis (i$^*$).

We have now constructed sieve weights $\lambda_1, \lambda_2$ and $\lambda_3$ which are bounded and satisfy \eqref{eq2.1}. Moreover, $\lambda_1$ and $\lambda_3$ satisfy hypothesis~(i$^*$), and hence, by the discussion near the end of Section~\ref{sec4.1}, also hypothesis~(i). Next we verify that the functions $\lambda_1, \lambda_2$ and $\lambda_3$ satisfy the hypotheses~(ii)--(v) required in the proof of Propostion~\ref{prop1}.

\medskip

\paragraph*{\em Hypothesis (ii)} The argument in Kumchev \cite[Lemma 5.5]{Kumc06} establishes (essentially) that every convolution of the form
\begin{equation}\label{eq6.7}
  \lambda(m) = \sum_{\substack{ rs = m\\ r \sim R}} \xi_r\psi(s,w),
\end{equation}
where $|\xi_r| \le \tau(r)^c$, $1 \le R \le x^{11/20}$, and $w \le \sqrt{2x/R}$, satisfies hypothesis (ii). Furthermore, by Kumchev \cite[Remark 5.1]{Kumc06}, the same argument applies to convolutions where $w$ is an (explicit) prime divisor of $r$---as in $\gamma_7$ above, for example. We can use this observation to verify hypothesis (ii) for $\gamma_i$'s and $\gamma_i^*$'s, except for $\gamma_4$. Finally, since $\gamma_4$ is the characteristic function of products $m = p_1p_2$, with $V < p_1 \le p_2$, we can rewrite it as
\begin{equation}\label{eq6.8}
  \gamma_4(m) = \sum_{V < p < x_1^{1/2}} \psi(m/p, w_p), \quad w_p = \sqrt{x_1/p},
\end{equation}
after which we can again use the above observation to verify hypothesis (ii).

\paragraph*{\em Hypothesis (iii)} By construction, the functions $\lambda_i$ are supported on integers free of prime divisors $p < z = x^{\sigma}$. Hence, this hypothesis holds.

\paragraph*{\em Hypotheses (iv) and (v)} Together, these two hypotheses state that $\lambda$ satisfies an analogue of the Siegel--Walfisz theorem for short intervals of lengths $\gg x^{\theta}$. We derive these hypotheses from the classical Siegel--Walfisz theorem (see Harman \cite[Lemma 2.7]{Harm07} or Liu and Zhan \cite[Theorem 1.13]{LiuZhan}) and the following estimate:
\begin{enumerate}
  \item [(iv*)] Given any fixed $A, B > 0$, any Dirichlet character $\chi$ modulo $q \le L^B$, and any subinterval $\mathcal I'$ of $\mathcal I$, one has
\begin{equation}\label{eq6.9}
  \sum_{m \in \mathcal I'} \lambda(m)\chi(m) = \frac {|\mathcal I'|}{2y_0} \sum_{|m - x| \le y_0} \lambda(m)\chi(m)
  + O \big( x^{\theta}L^{-A} \big),
\end{equation}
where $y_0 = x\exp\big( -3(\log x)^{1/3} \big)$.
\end{enumerate}

Comparing this statement with the conclusions of Lemmas 7.2 and 10.13 in Harman \cite{Harm07}, one sees that one may verify condition (iv*) using results from the study of primes in short intervals in Harman \cite{Harm07}. Indeed, \cite[Lemma 10.13]{Harm07} and a slight modification of \cite[Lemma 7.5]{Harm07} (to allow for the presence of characters) establish (iv*) for intervals $\mathcal I'$ of length $y \ge x^{3/4 + \eps}$ and for any arithmetic function $\lambda$ of the form \eqref{eq6.7} with $|\xi_r| \le \tau(r)^c$, $1 \le R \le x^{7/12}$, and $w \le 2x^{1/3}$. As with hypothesis (ii), we can use this observation to verify (iv*) for $\gamma_1, \gamma_5, \gamma_7, \gamma_9, \gamma_{10}, \gamma_{11}, \gamma_6^*, \gamma_7^*$, and $\gamma_9^*$. We remark that in dealing with $\gamma_{11}$ we use the inequality $p_1p_2p_3 \le U^{3/2} < x^{7/12}$, which holds on the assumption that $\sigma \le (2\theta - 1)/7$ and $\theta > 4/5$. Furthermore, the condition $p \le 2x^{1/3}$ is implicit in $\gamma_3$, so (iv*) holds for that sum too. Finally, to show that $\gamma_4$ satisfies (iv*), we make use of \eqref{eq6.8} and note that $w_p \le 2x^{1/3}$ because $V > x^{1/3}$.

We have now verified that (iv*) holds for $\lambda_1, \lambda_2$ and $\lambda_3$. Applying the Siegel--Walfisz theorem to the right side of \eqref{eq6.9}, we obtain hypothesis (iv). Similarly, we obtain hypothesis (v) by combining \eqref{eq6.9} and the following lemma, which follows from the Prime Number Theorem by the inductive argument in Harman \cite[\S A.2]{Harm07}.

\begin{lemma}\label{lem6.1}
  Let $2\le z < y \le z^c$, let $\psi(n,z)$ be defined by \eqref{eq1.8}, and let $\omega(u)$ be the continuous solution of the differential delay equation
  \[
    \begin{cases}
      (u\omega(u))'=\omega(u-1) & \text{if } u > 2, \\
      \omega(u)=u^{-1}     & \text{if } 1 < u \le 2.
    \end{cases}
  \]
  Then
  \[
    \sum_{m \le y} \psi(m,z) = \frac 1{\log z} \sum_{z < m \le y} \omega \left( \frac {\log m}{\log z} \right)
    + O \big( y\exp \big( -(\log y)^{1/2} \big)\big).
  \]
\end{lemma}

For example, we have
\[
  \sum_{m \in \mathcal I'} \gamma_3(m) = \frac {|\mathcal I'|}{2y_0} \sum_{U \le p \le V} \frac {1}{\log p} \sum_{|mp - x| \le y_0} \omega\left( \frac {\log m}{\log p} \right) + O\big( x^{\theta}L^{-A} \big).
\]
Using that
\[
  \omega\left( \frac {\log m}{\log p} \right) - \omega\left( \frac {\log (x/p)}{\log p} \right) \ll y_0x^{-1}
\]
when $|mp - x| \le y_0$, we deduce
\begin{align*}
  \sum_{m \in \mathcal I'} \gamma_3(m) &= |\mathcal I'| \sum_{U \le p \le V} \frac {1}{p\log p} \omega\left( \frac {\log (x/p)}{\log p} \right) + O\big( x^{\theta}L^{-A} \big) \\
  &= |\mathcal I'|L^{-1} \int_{1-\theta+2\sigma}^{\theta-4\sigma} \omega\left( \frac {1 - u}u \right) \, \frac {du}{u^2} + O\big( x^{\theta}L^{-A} \big),
\end{align*}
where the second step follows from the Prime Number Theorem by partial summation. The same technique can be used to evaluate the contributions to hypothesis (v) from the other functions $\gamma_i$ and $\gamma_i^*$ that appear in the definitions of $\lambda_1, \lambda_2$ and $\lambda_3$.

\medskip

Having verified that the arithmetic functions $\lambda_1, \lambda_2$ and $\lambda_3$ satisfy hypotheses (i)--(v), we can now apply the arguments in Sections \ref{sec3}--\ref{sec5} to establish inequalities \eqref{eq2.6} and \eqref{eq2.9}. Thus, to complete the proof of the theorem, it remains to show that the constant $C = C(\theta) = \kappa_1 - \kappa_2\kappa_3$ in \eqref{eq2.9} is positive when $8/9 < \theta < 1$ and $\sigma = (2\theta - 1)/7$. Let $\ell_j$, $1 \le j \le 11$, denote the constant (depending on $\theta$) in the asymptotic formula
\[
  \sum_{|m - x| \le y_0} \gamma_j(m) = 2\ell_jy_0L^{-1} + O\big( x\exp\big( -L^{1/2} \big) \big),
\]
and let $\ell_j^*$ be defined similarly in terms of $\gamma_j^*$. By \eqref{eq6.5}, \eqref{eq6.6} and the Prime Number Theorem, we have
\[
  \kappa_1 - \kappa_2 = 1 - \ell_8, \quad \kappa_3 = 1 + \ell_4 + \ell_5^* + \ell_{10}^*,
\]
where $\gamma_{10}^*(m) = \gamma_8^*(m) - \gamma_9^*(m)$. Note that the conditions $p_2 \ge z$ and $p_1p_2 < U$ in $\gamma_{10}^*$ make the condition $p_1 \le V^{1/2}$ superfluous; after that condition is dropped, we have $\gamma_{10}^* = \gamma_{11}$. Hence,
\[
  C = 1 - \ell_8 - \kappa_2(\kappa_2 + \ell_5^*),
\]
with
\begin{gather*}
  \kappa_2 = \log \left( \frac {3 + \theta}{4 - \theta} \right)
  + \iiint_{\mathcal D_{11}} \omega\left( \frac {1 - u - v - w}w \right) \, \frac {dudvdw}{uvw^2}, \\
  \ell_8 = \iint_{\mathcal D_8} \omega\left( \frac {1 - u - v}v \right) \, \frac {dudv}{uv^2}, \quad
  \ell_5^* = \int_{\theta/2-2\sigma}^{1-\theta+2\sigma} \omega\left( \frac {1 - u}u \right) \, \frac {du}{u^2},
\end{gather*}
where $\mathcal D_8$ is the two-dimensional region defined by
\[
  \sigma \le v \le u \le 1 - \theta + 2\sigma, \quad u + v \ge \theta - 4\sigma,
\]
and $\mathcal D_{11}$ is the three-dimensional region defined by
\[
  \sigma \le w \le v \le u < u + v \le 1 - \theta + 2\sigma, \quad u + v + w \ge \theta - 4\sigma.
\]

To prove the positivity of $C(\theta)$, it suffices to estimate the $\kappa_2, \ell_8$ and $\ell_5^*$ from above. To that end, we note that the Buchstab function $\omega$ is positive and satisfies (see Harman \cite[eq. (1.4.16)]{Harm07})
\begin{equation}\label{eq6.10}
  \omega(u) \le \begin{cases}
    1/u               & \text{if } 1 \le u \le 2, \\
    (1 + \log(u-1))/u & \text{if } 2 < u \le 3, \\
    (1 + \log 2)/3    & \text{if } u > 3.
  \end{cases}
\end{equation}
(In fact, this inequality is an equality when $1 \le u \le 3$.) Using \eqref{eq6.10} and numerical integration to estimate $C(\theta)$, we find that
\[
  C(\theta) \ge \tilde C(\theta),
\]
where $\tilde C(\theta)$ is an increasing function of $\theta$ with $\tilde C(8/9) > 0.17$ (see Figure \ref{fig1}). This completes the proof of the theorem.

\begin{figure}
  \ifpdf \boxed{\includegraphics{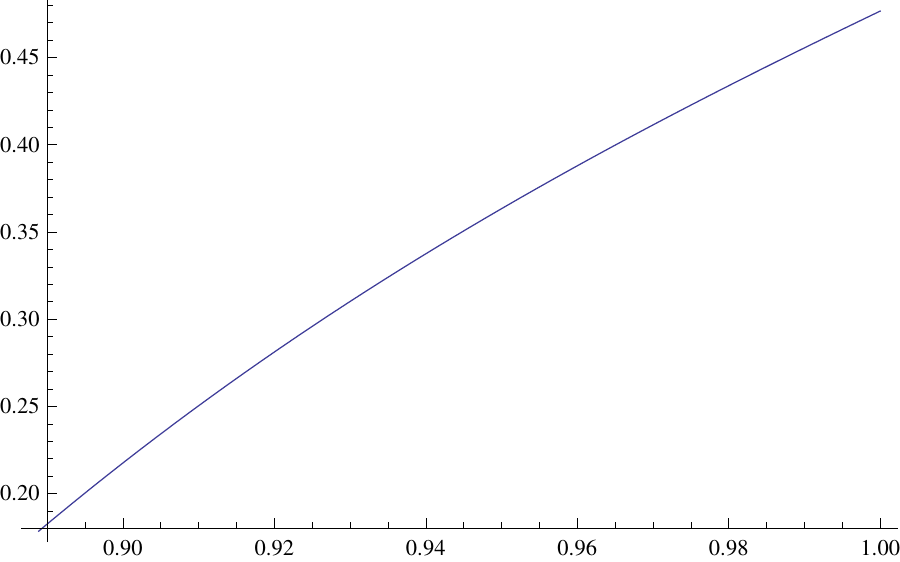}}
  \else  \boxed{\includegraphics{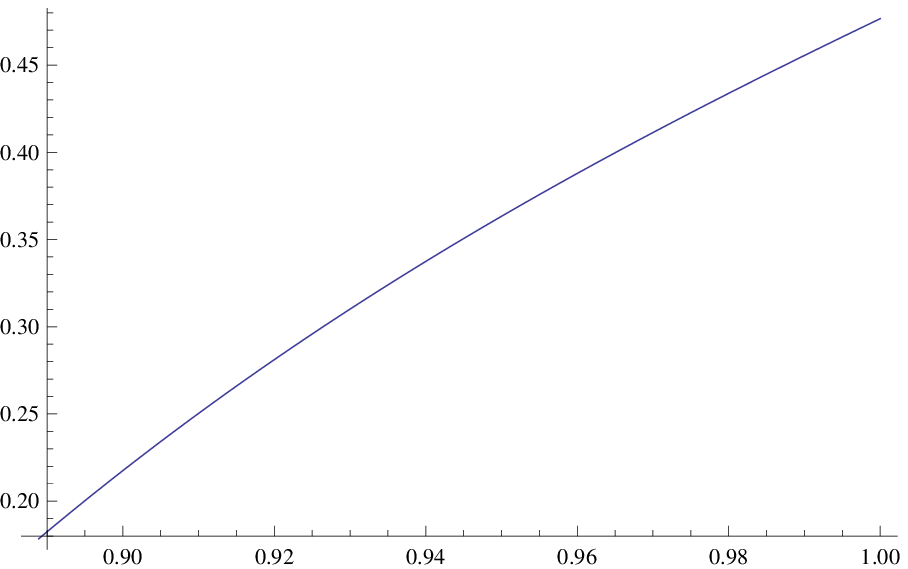}} \fi
  \hspace{.5in}
  \begin{tabular}[b]{cc}
  \toprule
  $\theta$    & $\tilde C(\theta)$ \\
  \midrule
  $1$    & $0.476$ \\
  $0.98$ & $0.433$ \\
  $0.96$ & $0.387$ \\
  $0.95$ & $0.363$ \\
  $0.94$ & $0.337$ \\
  $0.93$ & $0.310$ \\
  $0.92$ & $0.281$ \\
  $0.91$ & $0.250$ \\
  $0.90$ & $0.217$ \\
  $0.89$ & $0.182$ \\
  $8/9$  & $0.178$ \\
  \bottomrule
  \end{tabular}
  \caption{Graph and select values of $\tilde C(\theta)$, $8/9 \le \theta \le 1$.}\label{fig1}
\end{figure}

\section{The proofs of Theorems \ref{thm1}, \ref{thm3}, \ref{thm4}, and \ref{thm5}}
\label{sec7}

\subsection{Proof of Theorem \ref{thm1}}

Let $n$ be a large integer with $n \equiv 5 \pmod {24}$ and let $\theta = 8/9 + \eps$. We set $X = 4n/5$ and $H = X^{\theta/2}$. Then the set
\[
  \mathcal M = \big\{ n - p^2 \mid \big| p - (n/5)^{1/2} \big| < 0.5H \big\}
\]
contains $\gg X^{\theta/2}(\log X)^{-1}$ integers $m_p \in \mathcal H_4$ satisfying
\[
  |m_p - X| \le \big| p^2 - (n/5) \big| < 0.5H(1 + o(1))X^{1/2} < HX^{1/2}.
\]
Since $\theta/2 > (16-11\theta)/14 + \eps$, it follows from Theorem \ref{thm2} that there is an $m_p \in \mathcal M$ that can be represented as
\[
  m_p = p_1^2 + \dots + p_4^2, \qquad \big| p_j - (X/4)^{1/2} \big| \le H.
\]
Hence, $n$ can be represented
\[
  \begin{cases}
    n = p^2 + p_1^2 + \dots + p_4^2, & \\
    \big| p - (n/5)^{1/2} \big| < 0.5H, \; \big| p_j - (n/5)^{1/2} \big| < H. &
  \end{cases}
\]
This establishes Theorem \ref{thm1}.

\subsection{Proof of Theorem \ref{thm3}}

This proof shares many features with the proofs of Theorems \ref{thm4} and~\ref{thm5}, so we set the initial stages in more generality than is needed for Theorem \ref{thm3} itself. For $s \ge 3$ and a large $X$, we set
\[
  x = \sqrt{X/s}, \quad \mathcal I_s = \big( x-x^{\theta}, x + x^{\theta} \big],
\]
and we consider the sum
\[
  R_s(n) = \sum_{\substack{ p_1^2 + \dots + p_s^2 = n\\ p_i \in \mathcal I_s}} 1.
\]
Similarly to \eqref{eq2.3}, we have
\begin{equation}\label{eq7.1}
  R_s(n) = \bigg( \int_{\mathfrak M} + \int_{\mathfrak m} \bigg) f_0(\alpha)^s \mathrm e(-\alpha n) \, d\alpha,
\end{equation}
where $f_0(\alpha)$ is the exponential sum from Section \ref{sec2} (with the summation range $\mathcal I$ replaced by $\mathcal I_s$ when $s \ne 4$) and the sets $\mathfrak M$ and $\mathfrak m$ are defined in \eqref{eq2.4} and \eqref{eq2.5}. In particular, we introduce (\emph{via} \eqref{eq2.4}) a parameter $\sigma$, which will be specified shortly.)

When $s = 4$, Proposition \ref{prop1} with $j = k = 0$ yields
\begin{equation}\label{eq7.2}
  \int_{\mathfrak M}  f_0(\alpha)^4 \mathrm e(-\alpha n) \, d\alpha
  = \mathfrak S(n)\mathfrak I(n) \big( 1 + O \big( L^{-1} \big) \big),
\end{equation}
provided that $0 < \sigma \le \theta - 31/40$. Furthermore, using Lemma \ref{lem4.4} and \cite[Theorem 2]{LiZh96} (instead of Lemmas \ref{lem4.1}--\ref{lem4.3} above), we can show as in Section \ref{sec4.1} that
\begin{equation}\label{eq7.3}
  \sup_{\alpha \in \mathfrak m} |f_0(\alpha)| \ll x^{\theta - \sigma + \eps},
\end{equation}
provided that $0 < \sigma \le \min( (3\theta - 2)/6, (2\theta - 1)/8, (10\theta - 7)/15 )$. Let $Y = X^{(1+\theta)/2}$ and
\[
  \sigma =\min( \theta - 31/40, (2\theta - 1)/8 ).
\]
Using Bessel's inequality, \eqref{eq3.1} and \eqref{eq7.3}, we obtain
\begin{equation}\label{eq7.4}
  \sum_{|n - X| \le Y} \bigg| \int_{\mathfrak m}  f_0(\alpha)^4 \mathrm e(-\alpha n) \, d\alpha \bigg|^2 \le \int_{\mathfrak m} |f_0(\alpha)|^8 \, d\alpha \ll x^{6\theta - 4\sigma + \eps}.
\end{equation}
In particular, we have
\begin{equation}\label{eq7.5}
  \int_{\mathfrak m}  f_0(\alpha)^4 \mathrm e(-\alpha n) \, d\alpha \ll x^{3\theta - 1 - \eps}
\end{equation}
for all but $O\big( X^{1 - 2\sigma + \eps} \big)$ values of $n$ with $|n - X| \le Y$. Combining \eqref{eq7.1}, \eqref{eq7.2} and \eqref{eq7.5}, we conclude that
\begin{equation}\label{eq7.6}
  R_4(n) = \mathfrak S(n)\mathfrak I(n) \big( 1 + O \big( L^{-1} \big) \big)
\end{equation}
for all but $O\big( X^{1 - 2\sigma + \eps} \big)$ values of $n \in \mathcal H_4$ with $|n - X| \le Y$. Theorem \ref{thm3} follows from \eqref{eq7.6} and~\eqref{eq5.1a}.

\subsection{Proof of Theorem \ref{thm5}}

We start from \eqref{eq7.1} with $s \ge 6$, $\theta = \theta_s + \eps$, $\sigma = \theta - 31/40$, and $X = n$. With these choices, it is not difficult to generalize the argument of Proposition \ref{prop1} to show that
\begin{equation}\label{eq7.7}
  \int_{\mathfrak M}  f_0(\alpha)^s \mathrm e(-\alpha n) \, d\alpha
  = \mathfrak S_s(n)\mathfrak I_s(n) \big( 1 + O \big( L^{-1} \big) \big).
\end{equation}
Here, $\mathfrak S_s(n)$ and $\mathfrak I_s(n)$ are $s$-dimensional variants of the singular series and the singular integral from Section \ref{sec5} which satisfy the bounds
\begin{equation}\label{eq7.8}
  x^{(s-1)\theta - 1}L^{-s} \ll \mathfrak S_s(n)\mathfrak I_s(n) \ll x^{(s-1)\theta - 1}L^{-s}
\end{equation}
whenever $n \in \mathcal H_s$. On the other hand, by \eqref{eq3.1} and \eqref{eq7.3},
\begin{equation}\label{eq7.9}
  \int_{\mathfrak m} |f_0(\alpha)|^s \, d\alpha \ll x^{(s-2)\theta - (s-4)\sigma + \eps}.
\end{equation}
Combining \eqref{eq7.1} and \eqref{eq7.7}--\eqref{eq7.9}, we obtain
\[
  R_s(n) \gg x^{(s-1)\theta - 1}L^{-s},
\]
provided that
\begin{equation}\label{eq7.10}
  (s-4)\sigma > 1 - \theta.
\end{equation}
Since our choice of $\sigma$ and $\theta$ satisfies this inequality, this completes the proof of the theorem.

\subsection{Sketch of the proof of Theorem \ref{thm4}}

The proof of the first part of Theorem \ref{thm4} follows the same script as that of Theorem \ref{thm3}, so we only outline the necessary modifications. First, in place of~\eqref{eq7.4}, we have
\[
  \sum_{|n - X| \le Y} \bigg| \int_{\mathfrak m}  f_0(\alpha)^3 \mathrm e(-\alpha n) \, d\alpha \bigg|^2 \le \int_{\mathfrak m} |f_0(\alpha)|^6 \, d\alpha \ll x^{4\theta - 2\sigma + \eps}.
\]
This suffices to establish an appropriate version of \eqref{eq7.5} for all but $O\big( X^{1-\sigma + \eps} \big)$ values of $n$ with $|n-X| \le Y$.

We also need to be more careful when we approach the required version of \eqref{eq7.2}. A slight variant of the proof of Proposition \ref{prop1} yields (see the treatment of the major arcs in \cite{HaKu10}, especially, the way \cite[Lemma 1]{HaKu10} is applied in the proof of Lemma 4 of that paper)
\begin{equation}\label{eq7.11}
  \int_{\mathfrak M} f_0(\alpha)^3 \mathrm e(-\alpha n) \, d\alpha
  = \mathfrak S_3(n, P)\mathfrak I_3(n) \big( 1 + O \big( L^{-1} \big) \big).
\end{equation}
It is important that \eqref{eq7.11} holds under the same restrictions on $\theta$ and $\sigma$ as in Section \ref{sec5}, and it is to maintain those restrictions that we need the more subtle arguments from \cite{HaKu10}. Here, $\mathfrak I_3(n)$ is a three-dimensional singular integral similar to $\mathfrak I(n)$ above, and $\mathfrak S_3(n, P)$ is a partial sum of the singular series $\mathfrak S_3(n)$ for sums of three squares of primes. Unlike the case $s \ge 4$, this partial singular series poses some technical difficulties, but those have been resolved in prior work on the Waring--Goldbach problem for three squares (see Liu and Zhan \cite[Lemma 8.8]{LiuZhan} and Harman and Kumchev \cite[Lemma 7]{HaKu10}). In particular, by \cite[Lemma 7]{HaKu10}, we have
\[
  \mathfrak S_3(n,P)\mathfrak I_3(n) \gg x^{2\theta-1}L^{-6}
\]
for all but $O\big( X^{1-\sigma+\eps} \big)$ values of $n \in \mathcal H_3$ with $|n - X| \le Y$. This suffices to complete the proof of~\eqref{eq1.7}.

To establish the second part of Theorem \ref{thm4}, we have to make similar adjustments to the proof of Theorem \ref{thm2}. Leaving the sieve part of the argument unchanged, we then obtain a version of \eqref{eq1.7} with $\sigma = (2\theta - 1)/7$ and complete the proof of the theorem.

\section{Final remarks}
\label{sec8}

We conclude this paper with some remarks on the possibilities for further improvements on some of our theorems.
\begin{enumerate}
  \item The most attractive (but also most challenging) direction for further progress involves the restriction $\theta > 8/9$ in Theorem \ref{thm2}. In our work, that restriction is forced upon us by the assumption near the end of Section \ref{sec2} that $\sigma > 1 - \theta$. In our arguments, we require that $\sigma < (3\theta - 2)/6$ in several places in Section \ref{sec4.1} and that $\sigma \le (2\theta - 1)/7$ to ensure that hypothesis (iii) in Section \ref{sec5} holds. Either of those upper bounds on $\sigma$ in combination with the requirement $\sigma > 1 - \theta$ restricts $\theta$ to the range in Theorem \ref{thm2}. However, as explained below, the constraint $\sigma \le (2\theta - 1)/7$ can be eliminated, though at a considerable cost in terms of added complexity to the already involved argument in Section \ref{sec5}. Since some bounds for ``long'' quadratic Weyl sums over primes do not yet have analogues for sums over short intervals, it is possible that the restriction $\sigma < (3\theta - 2)/6$ in Lemmas \ref{lem4.3} and \ref{lem4.4} can also be relaxed. It seems that if one succeeds to relax that restriction, one should be able to reach values of $\theta$ below $8/9$.
  \item A possible improvement on Theorems \ref{thm2} and \ref{thm4}, one which does not require new exponential sum bounds, concerns the exponents $(16 - 11\theta)/14$ and $(8 - 2\theta)/7$ in those results. Those exponents are determined by the choice $\sigma = (2\theta - 1)/7$ in Section \ref{sec6}, which is made in accordance with the restriction $z \ge x^{\sigma}$ in hypothesis (iii) above. Since that restriction is imposed more for convenience than out of necessity, it is possible to dispense with it and to increase the value of $\sigma$ in the proof of Theorem \ref{thm2}, at least when $\theta$ is close to $1$. That should result in sharper bounds and may even close the small gap between our results with $\theta = 1 - \eps$ and \eqref{i.0}. We chose not to pursue this possibility here, because the omission of the assumption $z \ge x^{\sigma}$ from hypothesis (iii) unleashes a tidal wave of technical complications on the major arcs, with very little potential return. For a glimpse of those complications, the reader can compare the treatments of the major arcs in the papers of Harman and Kumchev \cite{HaKu06, HaKu10}.
  \item It is also possible to achieve small improvements on the estimates in Theorems \ref{thm3} and \ref{thm4} when $13/15 < \theta \le 8/9$. In this range, those theorems claim their respective bounds with $\sigma = (2\theta - 1)/8$, a value obtained by applying an exponential sum estimate of Liu and Zhan \cite[Theorem 2]{LiZh96}. It is possible to modify the sieve argument in Section \ref{sec6} so that it can be applied when $\theta \le 8/9$ and $0 < \sigma < (3\theta - 2)/6$. (When $\theta \le 8/9$, the latter condition implies $\sigma < (2\theta - 1)/7$, so this should not lead to major arc troubles.) One should then be able to leverage the modified sieve construction into a larger choice of $\sigma$, and hence, stronger upper bounds. We did not pursue this path, because we wanted to keep the combinatorial argument in Section \ref{sec6} relatively simple. However, taken on its own, the work involved in such an improvement should not be prohibitive, and we hope to see this minor flaw of our theorems corrected in the future.
\end{enumerate}

\begin{acknowledgment}
  This paper was written while the second author was visiting Towson University on a scholarship under the State Scholarship Fund
  by the China Scholarship Council. He would like to take this opportunity to thank the Council for the support and the
  Towson University Department of Mathematics for the hospitality and the excellent working conditions.
  The first author thanks Professors Jianya Liu and Guangshi L\"{u} and Shandong University at Weihai for their hospitality
  during a visit to Weihai in August of 2011, when a portion of this work was completed.
  Furthermore, both authors would like to express their gratitude to Professors Liu and L\"{u}
  for several suggestions and comments on the content of the paper.
  Last but not least, we would like to thank Professor Trevor Wooley for suggesting the question answered in Theorem \ref{thm5}.
\end{acknowledgment}


\begin{thebibliography}{99}

\bibitem{Baue97} %MR1452156
    \newblock C. Bauer,
    \newblock \emph{A note on sums of five almost equal prime squares},
    \newblock Arch. Math. (Basel) \textbf{69} (1997), 20--30.

\bibitem{Baue05} %MR2156959
    \newblock C. Bauer,
    \newblock \emph{Sums of five almost equal prime squares},
    \newblock Acta Math. Sin. (Engl. Ser.) \textbf{21} (2005), 833--840.

\bibitem{BaWa06} %MR2241603
    \newblock C. Bauer and Y. H. Wang,
    \newblock \emph{Hua's theorem for five almost equal prime squares},
    \newblock Arch. Math. (Basel) \textbf{86} (2006), 546--560.

\bibitem{ChKu06} %MR2232504
    \newblock S. K. K. Choi and A. V. Kumchev,
    \newblock \emph{Mean values of Dirichlet polynomials and applications to linear equations with prime variables},
    \newblock Acta Arith. \textbf{123} (2006), 125--142.

\bibitem{Harm07} %MR2331072
    \newblock G. Harman,
    \newblock \emph{Prime-Detecting Sieves},
    \newblock Princeton University Press, Princeton, 2007.

\bibitem{HaKu06} %MR???????
    \newblock G. Harman and A. V. Kumchev,
    \newblock \emph{On sums of squares of primes},
    \newblock Math. Proc. Cambridge Philos. Soc. \textbf{140} (2006), 1--13.

\bibitem{HaKu10} %MR2653209
    \newblock G. Harman and A. V. Kumchev,
    \newblock \emph{On sums of squares of primes II},
    \newblock J. Number Theory \textbf{130} (2010), 1969--2002.

\bibitem{Hua38} %MR
    \newblock L. K. Hua,
    \newblock \emph{Some results in the additive prime number theory},
    \newblock Q. J. Math. (Oxford) \textbf{9} (1938), 68--80.

\bibitem{Kumc06} %MR2653209
    \newblock A. V. Kumchev,
    \newblock \emph{On Weyl sums over primes and almost primes},
    \newblock Michigan Math. J. \textbf{54} (2006), 243--268.

\bibitem{LiWu08} %MR???????
    \newblock H. Z. Li and J. Wu,
    \newblock \emph{Sums of almost equal prime squares},
    \newblock Funct. Approx. Comment. Math. \textbf{38} (2008), 49--65.

\bibitem{LiLvZh06} %MR2250891
    \newblock J. Y. Liu, G. S. L\"{u} and T. Zhan,
    \newblock \emph{Exponential sums over primes in short intervals},
    \newblock Sci. China Ser. A \textbf{49} (2006), 611--619.

\bibitem{LiZh96} %MR1414517
    \newblock J. Y. Liu and T. Zhan,
    \newblock \emph{On sums of five almost equal prime squares},
    \newblock Acta Arith. \textbf{77} (1996), 369--383.

\bibitem{LiZh98} %MR1633783
    \newblock J. Y. Liu and T. Zhan,
    \newblock \emph{Sums of five almost equal prime squares II},
    \newblock Sci. China Ser. A \textbf{41} (1998), 710--722.

\bibitem{LiZh00} %MR1813461
    \newblock J. Y. Liu and T. Zhan,
    \newblock \emph{Hua's theorem on prime squares in short intervals},
    \newblock Acta Math. Sin. (Engl. Ser.) \textbf{16} (2000), 669--690.

\bibitem{LiuZhan} %
    \newblock J. Y. Liu and T. Zhan,
    \newblock \emph{New Developments in the Additive Theory of Prime Numbers},
    \newblock World Scientific, Singapore, 2012.

\bibitem{Lv05} %MR2143653
    \newblock G. S. L\"{u},
    \newblock \emph{Hua's theorem with five almost equal prime variables},
    \newblock Chinese Ann. Math. Ser. B \textbf{26} (2005), 291--304.

\bibitem{Lv06} %MR2220183
    \newblock G. S. L\"{u},
    \newblock \emph{Hua's theorem on five almost equal prime squares},
    \newblock Acta Math. Sin. (Engl. Ser.) \textbf{22} (2006), 907--916.

\bibitem{LvZh09} %
    \newblock G. S. L\"{u} and W. G. Zhai,
    \newblock \emph{On the representation of an integer as sums of four almost equal squares of primes},
    \newblock Ramanujan J. \textbf{18} (2009), 1--10.

\bibitem{Meng06} %MR2252022
    \newblock X. M. Meng,
    \newblock \emph{The five prime squares theorem in short intervals},
    \newblock Acta Math. Sin. (Chin. Ser.) \textbf{49} (2006), 405--420.

\bibitem{MoVa75}
    \newblock H. L. Montgomery and R. C. Vaughan,
    \newblock \emph{The exceptional set in Goldbach's problem},
    \newblock Acta Arith. \textbf{27} (1975), 353--370.

\bibitem{Schw61}
    \newblock W. Schwarz,
    \newblock \emph{Zur Darstellung von Zahlen durch Summen von Primzahlpotenzen},
    \newblock J. Reine Angew. Math. \textbf{206} (1961), 78--112, in German.

\bibitem{Wool02} %MR1435742
    \newblock T. D. Wooley,
    \newblock \emph{Slim exceptional sets for sums of four squares},
    \newblock Proc. London Math. Soc. (3) \textbf{85} (2002), 1--21.

\end{thebibliography}
\end{document}